\documentclass[reqno,a4paper]{amsart} 
\usepackage{amsmath}    
\usepackage{amssymb}    
\usepackage{amsfonts}
\usepackage{graphicx}

\newtheorem{theorem}{Theorem}[section]

\newtheorem{lemma}[theorem]{Lemma}
\newtheorem{corollary}[theorem]{Corollary}

\theoremstyle{remark}
\newtheorem{remark}[theorem]{Remark}

\numberwithin{equation}{section}

\newcommand{\norm}[1]{\|#1\|}

\newcommand{\Tnorm}[1]{|\mkern-1mu|\mkern-1mu|#1|\mkern-1mu|\mkern-1mu|}

\newcommand{\Mnorm}[1]{|\mkern-1mu|\mkern-1mu|#1|\mkern-1mu|\mkern-1mu|_\mathrm{M}}
\newcommand{\Vnorm}[1]{|\mkern-1mu|\mkern-1mu|#1|\mkern-1mu|\mkern-1mu|_\mathrm{V}}
\newcommand{\sprod}[2]{\langle #1, #2 \rangle}
\def\d{\partial}

\def\B{\mathcal{B}}
\def\L{\mathcal{L}}
\def\Bt{\tilde{\mathcal{B}}}
\def\Lt{\tilde{\mathcal{L}}}
\def\R{{\mathbb R}}

\def\Kt{\tilde{K}}

\def\et{\tilde{e}}
\def\etat{\tilde{\eta}}
\def\xit{\tilde{\xi}}

\def\sumK{\sum_{K\in\mathcal{\tilde{C}}_h}\!\!}
\def\sumKV{\sum_{K\in\mathcal{C}_h}\!\!}
\def\curl{\nabla_x \times }
\def\gradx{\nabla_x}
\def\fii{\varphi}
\def\dbart{\bar{\partial}_t}
\def\dbartt{\bar{\partial}_t^2}
\def\EE{\mathcal{E}}
\def\bbE{\mathbb{E}}
\def\bv{\mathbf{v}}
\def\bx{\mathbf{x}}

\begin{document}
\title[hp-SD Schemes for the Vlasov-Maxwell System]
{On hp-Streamline Diffusion and Nitsche schemes for the Relativistic Vlasov-Maxwell System}

\author[M.~Asadzadeh]{M. Asadzadeh}

\address{ 
Department of Mathematics, 
Chalmers University of Technology and  G\"oteborg University, 
SE--412 96, G\"oteborg, Sweden
} 
\email{mohammad@chalmers.se}



\author[P.~Kowalczyk]{P. Kowalczyk}

\address{
Institute of Applied Mathematics and Mechanics,
University of Warsaw,
Banacha 2, 02-097 Warszawa, Poland
}
\email{pkowal@mimuw.edu.pl} 

\author[C.~Standar]{C. Standar}

\address{
Department of Mathematics,
Chalmers University of Technology and  G\"oteborg University,
SE--412 96, G\"oteborg, Sweden
}
\email{standarc@chalmers.se}


\begin{abstract}  
We  study stability and convergence of 
 $hp$-streamline diffusion (SD) finite element, 
and Nitsche's schemes for the three dimensional,  
relativistic (3 spatial dimension and 3 velocities), time dependent 
 Vlasov-Maxwell system 
and Maxwell's equations, respectively.  
For the $hp$ scheme for the Vlasov-Maxwell system, 
assuming that the exact solution is in the Sobolev space 
$H^{s+1}(\Omega)$, we derive global {\sl a priori} error bound of 
order ${\mathcal O}(h/p)^{s+1/2}$, 
where $h (= \max_K h_K)$ is the mesh parameter and $p (= \max_K p_K)$ 
is the spectral order. 
This estimate is based on the local version with $h_K=\mbox{ diam } K$ 
being the diameter of the {\sl phase-space-time} element $K$ and 
$p_K$ is the spectral order (the degree of approximating finite element 
polynomial) for $K$. 
As for the 
Nitsche's scheme, by a simple calculus of the field equations,  
first we convert the Maxwell's system to an {\sl elliptic type} 
 equation. 
Then, combining the Nitsche's method for the spatial discretization 
with a second order time scheme, we obtain optimal convergence of 
 ${\mathcal O}(h^2+k^2)$, where $h$ is the spatial mesh size and 
$k$ is the time step. Here, as in the classical literature, the second 
order time scheme requires higher order regularity assumptions. 

Numerical justification of the results, in 
lower dimensions, is presented
and is also the subject of a forthcoming computational work 
\cite{Malmberg_Standar}.

\end{abstract}

\maketitle

\section{Introduction}

We study stability and convergence for some specific finite element 
schemes for a model problem for the three dimensional, 
relativistic, Vlasov-Maxwell (VM) system 
with 3-dimensional spatial domain ($\bx\in\Omega_x\subset{\mathbb R}^3$) 
and 3-dimensional velocities domain ($\bv \in\Omega_v\subset{\mathbb R}^3$).  
The objective is two-fold: 

i)  Numerical investigations of the $hp$ 
-version of the streamline diffusion (SD)
finite element method for VM where both Maxwell's 
and Vlasov equations are discretized using a 
space-velocity-time scheme both in $h$ (mesh size) and 
in $p$ (spatial order) versions. In this part we derive optimal 
a priori error bounds for a SD scheme in a
$L_2$-based norm.

ii) 
The study of the combined 
effect of Nitsche's symmetrization (cf \cite{Nitsche} and \cite{AssousMichaeli})
in the spatial scheme for a Galerkin method 
and a time discretization, for a second order pde obtained through 
the combined Maxwell's fields. 

The SD method was suggested by Hughes and Brooks in \cite{Hughes_Brooks:79} 
for the 
fluid problems. The method was further developed  
(by T. Houghs and co-workers) to include several engineering problems. 
A through mathematical analysis was first given by Johnson et al in 
\cite{John} in a study of the Navier-Stokes equations and was 
extended to most pdes with hyperbolic nature, where, e.g., 
\cite{Asad}-\cite{AsadSopas} and \cite{Szepessy} 
are relevant in the present study. 
In the SD method the test function is modified by adding       a multiple 
of the streaming part in the equation, in terms of the test function, to it. 
Then, in the weak formulation we obtain a multiple of streaming 
terms in test and trial functions. 
This can be viewed as an extra diffusion term in the streaming direction in 
the original equation. Hence, the name of the method 
(the streamline diffusion). Such an extra diffusion would improve both the 
stability and convergence properties of the underlying Galerkin scheme. 
It is well known that the standard Galerkin method has a weaker convergence 
property for the 
hyperbolic problems: ${\mathcal O}(h^{s-1})$ versus  
${\mathcal O}(h^{s})$ for the elliptic and parabolic problems with  
exact solution in the Sobolev space $H^s(\Omega)$. 
The SD method improves this weak convergence to ${\mathcal O}(h^{s-1/2})$ 
(see \cite{Adams} for Sobolev spaces of non-integer order) 
and also, having an upwinding character, enhances the stability. 

These two properties are achieved by discontinuous Galerkin as well (see, e.g. \cite{Bre}).
The $hp$-approach is to capture local behavior in the sense that: in 
the vicinity of singularities refined mesh $h$ is combined with the 
lower order (small $p$) polynomial approximations, whereas in more smooth 
regions higher order polynomials (large $p$) and non-refined (large $h$) meshes 
are used. In a sense the $hp$-approach may be interpreted as a kind of 
automatic adaptivity. 

The Vlasov-Maxwell (VM) system which describes the time evolution of 
collisionless 
plasma is formulated as 
\begin{equation}\label{VMI}
\begin{aligned}
\d_t f &+\hat v\cdot \nabla_x f+q(E+c^{-1} \hat{v}\times B)\cdot \nabla_v f=0,\\
&\d_t E=c \nabla_x \times B-j,\qquad \nabla_x \cdot E=\rho, \\
&\d_t B=-c\nabla_x \times E, \qquad  \nabla_x \cdot B=0
\end{aligned}
\end{equation}
with properly assigned initial data
$f(0,x,v)=f^0(x,v)\ge0$, $E(0,x)=E^0(x)$, $B(0,x)=B^0(x)$.
Here $f$ is the density, in phase space, time of 
particles with charge $q$,
mass $m$ and velocity 
$$
\hat v=(m^2+c^{-2}\vert v\vert^2)^{-1/2} v\qquad (v\,\,  \mbox{is momentum}).
$$
Further, $c$ is the speed of light and the charge and current densities $\rho$
and $j$ are given by 
\begin{equation*} 
\rho(t,x)=4\pi\int qf\, dv \quad \mbox{and} \quad j(t,x)= 4\pi\int qf\hat v\, dv. 
\end{equation*}
The phase-space variables may have different dimension: 
$(x,v)\in {\mathbb R}^d\times {\mathbb R}^{d^\prime}, \, d \leq d^\prime $.

The Vlasov-Maxwell equations arises in several branches of continuum physics, 
e.g. in astrophysics or rarefied gas dynamics. 
The main assumption underlying the model is that collisions are 
rare and therefore negligible. In this setting the above system describes the 
motion of a collisionless plasma, e.g., a high-temperature, 
low-density, ionized gas. 

For a thorough mathematical study of 
 VM models we refer to 
DiPerna and Lions \cite{DiPerna_Lions} 
and a most recent work by Glassey and co-workers 
\cite{Glassey:2016}-\cite{Glassey:90} and 
the references therein. The results in \cite{Glassey:2016} are for 
a lower dimensional model where the interest lies in classical solutions, 
and are based on compactness and regularity assumptions on the 
initial density and fields.

The main mathematical 
concern in dealing with the Vlasov-Maxwell system is related to the 
nonlinear term $(E+\hat v \times B)\cdot\nabla_v f$ which can be written in 
the divergence free form, viz. 
$(E+\hat v \times B)\cdot\nabla_v f={\mbox div}_v\Big((E+\hat v \times B)f\Big)$.
In \cite{DiPerna_Lions} the nonlinear form $(E+\hat v \times B) f$ is analyzed. 

Numerical approaches for the VM system 
 have been considered by several authors in different setting. 
The most relevant studies to this work are given by 
Gamba and co-workers \cite{Che} devoted to 
a discontinuous Galerkin approach, and Standar in \cite{Standar:2016} where 
the stability and a priori error estimates for  
the $h$ version of SD method for VM are derived. 
As some related studies we mention the  analysis of 
a one dimensional model problem for the relativistic 
VM system in an interval given by Filbet and co-workers in 
\cite{Filbet_etal}. Also in a very recent work \cite{Degond_etal} 
Degond and co-workers study  a particle-in-cell method 
for the Vlasov-Maxwell system. 

An outline of this paper is as follows. 
We gather notation and assumptions in Section 2. 
In Section 3 we formulate the SD schemes for both Maxwell's equations and the 
Vlasov-Maxwell system. Section 4 is on stability and convergence of the 
$hp$ SD finite element method of the Maxwell's equations, based on a
space-time iterative scheme. We insert such approximated field function 
in the drift term in Vlasov-Maxwell equation and prove stability and 
derive optimal convergence rates in the SD 
phase-space-time discretization scheme. Section 5 is devoted to a 
the study of a Nitsche 
scheme combined with a  time discretization 
for a second order pde obtained from the Maxwell's equations. 
Here, we rely on a modified form of the Ritz projection and derive optimal 
error estimates for the Nitsche scheme in spatial discretization. 
Assuming additional regularities in time we also prove optimal convergence of 
a second order time scheme for the fields. 
Finally, in our concluding Section 6 we present some numerical tests
of the studied schemes in lower dimensional geometry.

\section{Notation and Assumptions}

The divergence equations in \eqref{VMI} can be derived from the rest of 
the equations, assuming that the initial data $E^0$ and $B^0$ satisfy
corresponding divergence equations.
Hence we will consider the following relativistic 
Vlasov-Maxwell system in $\R^d$
(in the paper we focus on the dimension $d=3$, but one can easily obtain
the analogous results for $d=2$)
\begin{equation}\label{VM_main}
\begin{aligned}
\d_t f &+\hat v\cdot \nabla_x f+(E+\hat{v}\times B)\cdot \nabla_v f=0,\\
&\d_t E=\nabla_x \times B-j,\\
&\d_t B=-\nabla_x \times E
\end{aligned}
\end{equation}
with $\hat v=(1+\vert v\vert^2)^{-1/2} v$ and $j(t,x)= \int f\hat v\, dv$,
where for simplicity we set the charge $q$ and all constants equal to one.

Our objective is to use an iterative scheme to approximate the solution of
the Vlasov-Maxwell (henceforth referred as VM) equations. 
First we take a guess for the density $ f $
and then calculate the corresponding $ j $. 
Next, we plug these quantities into the Maxwell's equations
and solve these equations. Finally, 
we solve the Vlasov equation with the such approximated
$ E $ and $ B $ as coefficients.

We start from the Maxwell's part.
Set $E=(E_1,E_2,E_3)^T\!$, $B=(B_1,B_2,B_3)^T\!$, $j=(j_1,j_2,j_3)^T\!$.
Then the Maxwell's equations in \eqref{VM_main} can be written in the 
following form:
$$
\begin{aligned}
\d_t E_1 &= \d_2 B_3-\d_3 B_2 -j_1, \\
\d_t E_2 &= \d_3 B_1-\d_1 B_3 -j_2, \\
\d_t E_3 &= \d_1 B_2-\d_2 B_1 -j_3, \\
\d_t B_1 &= -\d_2 E_3+\d_3 E_2, \\
\d_t B_2 &= -\d_3 E_1+\d_1 E_3, \\
\d_t B_3 &= -\d_1 E_2+\d_2 E_1, \\
\end{aligned}
$$
where $\d_i$ denotes the derivative with respect to $x_i$.
Hence defining the matrices 
$$
M_1= \left[ \begin{array}{cccccc}
0 & 0 & 0 & 0 & 0 & 0 \\
0 & 0 & 0 & 0 & 0 & 1 \\
0 & 0 & 0 & 0 & -1& 0 \\
0 & 0 & 0 & 0 & 0 & 0 \\
0 & 0 & -1& 0 & 0 & 0 \\
0 & 1 & 0 & 0 & 0 & 0 \\
\end{array} \right] , \quad
M_2= \left[ \begin{array}{cccccc}
0 & 0 & 0 & 0 & 0 & -1\\
0 & 0 & 0 & 0 & 0 & 0 \\
0 & 0 & 0 & 1 & 0 & 0 \\
0 & 0 & 1 & 0 & 0 & 0 \\
0 & 0 & 0 & 0 & 0 & 0 \\
-1& 0 & 0 & 0 & 0 & 0 \\
\end{array} \right] ,
$$
$$
M_3= \left[ \begin{array}{cccccc}
0 & 0 & 0 & 0 & 1 & 0 \\
0 & 0 & 0 & -1& 0 & 0 \\
0 & 0 & 0 & 0 & 0 & 0 \\
0 & -1& 0 & 0 & 0 & 0 \\
1 & 0 & 0 & 0 & 0 & 0 \\
0 & 0 & 0 & 0 & 0 & 0 \\
\end{array} \right]
$$
and letting $W=(E_1,E_2,E_3,B_1,B_2,B_3)^T\!$,
$W^0=(E_1^0,E_2^0,E_3^0,B_1^0,B_2^0,B_3^0)^T\!$ 
and $b=(-j_1,-j_2,-j_3,0,0,0)^T\!$,
the Maxwell's equations can be written as the system 
\begin{equation} \label{maxwell}
\left\{ \begin{array}{ll}
\d_t W + M_1 \d_1 W + M_2 \d_2 W + M_3 \d_3 W=b, \\
W(0,x)=W^0(x).
\end{array} \right.
\end{equation}

Now we return to the Vlasov equation given by
\begin{equation} \label{vlasov}
\left\{ \begin{array}{ll}
\d_t f +\hat v\cdot \nabla_x f+(E+\hat{v}\times B)\cdot \nabla_v f=0,\\
f(0,x,v) = f^0 (x,v) \geq 0.
\end{array} \right. 
\end{equation}
For simplicity, we introduce the notation
$$
G(f)=(\hat{v}, E+\hat{v}\times B)
$$
and define the total gradient 
$$
\nabla f = ( \nabla_x f , \nabla_v f ), 
$$
so that, we can rewrite the Vlasov equation in compact form as 
$$
\d_t f + G(f) \cdot \nabla f =0 .
$$
Note that $G$ is divergence free
\begin{equation*}
\nabla G(f)=\sum_{i=1}^d\frac{\d\hat{v}}{\d x_i}
+\sum_{i=d+1}^{2d}\frac{\d(E+\hat{v}\times B)}{\d v_{i-d}}
=\nabla_v\big(\hat{v}\times B\big)=0.
\end{equation*}

Throughout this paper $C$ will denote a generic constant, 
not necessarily the same at each occurrence, and independent of the parameters 
in the equations, unless otherwise explicitly specified.

\section{hp-Streamline Diffusion Method}

Let $\Omega_x\subset\R^3$ and $\Omega_v\subset\R^3$ denote the space 
and velocity domains,
respectively. We assume that $f(t,x,v)$, $E_i(t,x)$, 
$B_i(t,x)$ for $i=1,2,3$ have compact
supports in $\Omega_x$ and that $f(t,x,v)$ has compact support in $\Omega_v$.

Now we will introduce a finite element structure on 
$\Omega=\Omega_x\times\Omega_v$.
Let $T_h^x = \{ \tau_x \}$ and $T_h^v = \{ \tau_v \}$ 
be finite element subdivisions
of $\Omega_x$ with elements $\tau_x$ and $\Omega_v$ with elements 
$\tau_v$, respectively.
Then 
$T_h = T_h^x \times T_h^v = \{ \tau_x \times \tau_v \} = \{ \tau \}$ 
is a subdivision
of $\Omega$.
Let $0= t_0 < t_1 < \ldots < t_{M-1} < t_M=T$ be a partition of $[ 0, T ]$ into
sub-intervals $I_m = (t_m, t_{m+1} ]$, $m= 0, \ldots , M-1$.
Further let $\mathcal{C}_h$ be the corresponding subdivision of 
$Q_T=[0,T]\times\Omega$
into elements $K=I_m\times\tau$, with 
$h_K=\textnormal{diam}\,K$ as the mesh parameter.
We also define a piecewise constant mesh function $
h(t,x,v):=h_K$, $(t,x,v)\in K$. Finally, we 
introduce $\mathcal{\tilde{C}}_h$ as the finite element subdivision of 
$[0,T]\times\Omega_x$. 

\begin{remark}
Henceforth, the discrete problems are the  
finite element approximations for the equations \eqref{maxwell} and 
\eqref{vlasov} formulated for 
$(x,v,t)\in (0, T] \times { \Omega_x}\times {\Omega_v}$, 
associated with initial and corresponding boundary data. 
Here one may assume that $f$ has compact support in the 
velocity space $ {\mathbb R}_v^d$, and hence assume homogeneous 
Dirichlet boundary 
condition for $\Omega_v$. 
\end{remark}

Thus to define an adequate finite element space we let 
\[
\mathcal{H}_0 = \prod_{m=0}^{M-1} H_0^1 
(I_m \times \Omega_x \times \Omega_v)\quad 
\mbox{and} \quad \mathcal{\tilde{H}}_0 = 
\prod_{m=0}^{M-1} H^1_0 ( I_m \times \Omega_x ),
\]
where
\[
H^1_0 (I_m \times \Omega) = \{ w \in H^1 ; 
w= 0 \, \, \textnormal{on} \,\, \d \Omega \}.
\]
Here $ \Omega $ stands for either $ \Omega_x $ or 
$ \Omega_x \times \Omega_v $. For $ k=0,1,2, \ldots $, 
we define the finite element spaces for the Maxwell's 
equations (resp. Vlasov equation) as the space of piecewise polynomials which are 
continuous in $x$ (resp. in $x$ and $v$) and with possible discontinuities at the 
interior time levels $t_m, \, m=1, \ldots , M$:
\[ 
\tilde{V}_h= \{ g \in [\mathcal{\tilde{H}}_0]^6 ;
g_i|_{\Kt} \in P_{p_{\Kt}} (I_m) \times P_{p_{\Kt}} (\tau_x ) ,
\, \forall \Kt = I_m \times \tau_x \in \mathcal{\tilde{C}}_h , \, 1\le i\le 6\},
\]
with the extension for the Vlasov part and with 
$\tau=\tau_x\times\tau_v$, viz:
\[ 
{V}_h= \{ g \in \mathcal{H}_0 ; g|_{K} 
\in P_{p_K} (I_m) \times P_{p_K} (\tau_x )\times P_{p_K} (\tau_v)  , \, \forall K = I_m \times \tau \in \mathcal{C}_h \}.
\]
where $ P_{p_K} ( \cdot ) $ is the set of polynomial of degree at most 
$ p_K $ on the given set. In this setting we allow the degree of polynomial
to vary from cell to cell, hence we define the piecewise constant function
$p(t,x,v):=p_K$.
We shall also use some notation, viz.
\[
(f,g)_m=(f,g)_{S_m}, \qquad \| g \|_m =(g,g)_m^{1/2}
\]
and
\[
\sprod{f}{g}_m=(f(t_m, \ldots), g(t_m, \ldots))_{\Omega}, \qquad |g|_m=
\sprod{g}{g}_m^{1/2},
\]
where $ S_m = I_m \times \Omega $, is the slab at $ m $-th level, 
$ m=0, \ldots, M-1 $, and $\Omega $ stands for $\Omega_x$
in Maxwell's equations and $\Omega_x\times \Omega_v$ for the Vlasov case.

\subsection{Maxwell Equations} \label{secmax}

Define $f^{h,i}$, $b^{h,i}$ and $W^{h,i}$ as the approximation on the $i$th step
of $f$, $b$ and $W$, respectively.
The global \textit{hp} version of the streamline diffusion method on 
the $i$th step
for the Maxwell's part can now be formulated as follows: 
find $W^{h,i}\in\tilde{V}_h$ such that for $m=0, 1, \ldots , M-1$,
\begin{multline}\label{sdm1}
\Big(\d_t W^{h,i}+\sum_{l=1}^3 M_l\d_l W^{h,i},
g+\delta (\d_t g+\sum_{l=1}^3 M_l\d_l g)\Big)_m + \sprod{W^{h,i}_+}{g_+}_m= \\
= \Big(b^{h,i-1}, g+\delta (\d_t g+\sum_{l=1}^3 M_l\d_l g)\Big)_m 
+ \sprod{W^{h,i}_-}{g_+}_m,
\quad \forall \, g\in \tilde{V}_h,
\end{multline}
where $g = (g_1,\ldots,g_6)^T $, $ g_\pm (t,x) = 
\lim_{s \rightarrow 0^\pm} g(t+s, x)$. 
The problem \eqref{sdm1} is equivalent to: 
find $W^{h,i}\in\tilde{V}_h$ such that
\begin{equation} \label{sdm2}
\Bt (W^{h,i}, g) = \Lt (b^{h,i-1};g) \quad \forall \, g \in \tilde{V}_h,
\end{equation}
where the bilinear form is defined as
\begin{multline*}
\Bt (W,g)= \sum_{m=0}^{M-1} \Big(\d_t W + \sum_{l=1}^3 M_l\d_l W,
g + \delta (\d_t g + \sum_{l=1}^3 M_l\d_l g)\Big)_m \\
+ \sum_{m=1}^{M-1} \sprod{[ W ]}{g_+}_m + \sprod{W_+}{g_+}_0
\end{multline*}
and the linear form by
$$
\Lt (b;g) = 
\sum_{m=0}^{M-1} \Big(b, g + \delta (\d_t g +\sum_{l=1}^3 M_l\d_l g)\Big)_m
+ \sprod{W^0}{g_+}_0,
$$
where $ [ W ] = W_+ - W_- $.

Now let $(\cdot,\cdot)_K$ denote the $L_2$-inner product over K
and define a non-negative piecewise constant function $\delta$ by
$$
\delta|_K=\delta_K, \quad \mbox{for } K\in\mathcal{\tilde{C}}_h,
$$
i.e.,  $\delta_K$ is a non-negative constant on element $K$.
Counting for the local character of the parameters $h_K$, $p_K$ and $\delta_K$,
to formulate a finite element method based on the local 
space-time elements, the problem 
 \eqref{sdm2} would have an alternative formulation where we 
replace in the definitions
for $\Bt$ and $\Lt$ the sum of the inner products $(\cdot,\cdot)_m$
involving $\delta_K$ by the corresponding sum $\sumK{}(\cdot,\cdot)_K$
and all $\delta$ by $\delta_K$.
Thus we have the problem \eqref{sdm2} where in the bilinear, and linear,
 forms the first sum is replaced by $\sum_{K\in\tilde{\mathcal C}_h}$ 
as, e.g.,  
\begin{multline*}
\Bt(W,g)=\sumK \Big(\d_t W + \sum_{l=1}^3 M_l\d_l W,
g + \delta_K (\d_t g + \sum_{l=1}^3 M_l\d_l g)\Big)_K \\
+ \sum_{m=1}^{M-1} \sprod{[ W ]}{g_+}_m + \sprod{W_+}{g_+}_0
\end{multline*}
We also have that the solution $W$ of equation \eqref{maxwell} satisfies
$$
\Bt(W, g) = \Lt(b;g) \quad \forall \, g \in \tilde{V}_h.
$$
Subtracting \eqref{sdm2} from this equation,
we end up with the following relation 
\begin{equation} \label{ortomax}
\Bt (W-W^{h,i} , g) = \Lt (b;g) -\Lt (b^{h,i-1}; g) 
\quad \forall \, g \in \tilde{V}_h,
\end{equation}
which is of vital importance in the error analysis.

Now assuming jump discontinuities at the time levels 
$t=t_m,\, \,m=1,\ldots, M-1$, the suitable norm for 
 stability and convergence would read as follows: 
$$
\Mnorm{g}^2 = \frac{1}{2} 
\Big( | g_+ |^2_0 + | g_- |^2_M+\sum_{m=1}^{M-1} |[g]|^2_m
+2\!\sumK\delta_K\|\d_t g+\sum_{l=1}^3 M_l\d_l g\|_K^2 \Big).
$$

\subsection{Vlasov-Maxwell Equations}

The hp-streamline diffusion method on the $i$-th step
for the Vlasov part \eqref{vlasov} can be formulated as follows: 
find $f^{h,i}\in V_h$ such that for $m=0,1,\ldots,M-1$,
\begin{multline}\label{sdv1}
\Big(\d_t f^{h,i}+G(f^{h,i-1})\cdot\nabla f^{h,i},
g+\delta(\d_t g+G(f^{h,i-1})\cdot\nabla g)\Big)_m\\
+\sprod{f^{h,i}_+}{g_+}_m = \sprod{f^{h,i}_-}{g_+}_m \quad \forall g\in V_h.
\end{multline}
The problem \eqref{sdv1} is equivalent to: find $f^{h,i}\in V_h$ such that
\begin{equation}\label{sdv2}
\B(G(f^{h,i-1});f^{h,i},g) = \L(g) \quad \forall g\in V_h,
\end{equation}
where the trilinear form $\B$ is defined as
\begin{equation}\label{3LinForm}
\begin{split}
\B(G;f,g)
=&\sum_{m=0}^{M-1}\Big(\d_t f+G\cdot\nabla f,g
+\delta(\d_t g+G(f^{h,i-1})\cdot\nabla g)\Big)_m\\
&+\sum_{m=1}^{M-1}\sprod{[f]}{g_+}_m +\sprod{f_+}{g_+}_0
\end{split}
\end{equation}
and the linear form $L$ is given by
\[
\L(g)=\sprod{f^0}{g_+}_0.
\]

Analogously as for the Maxwell's equations we reformulate 
 \eqref{sdv2} considering phase-space-time finite element 
discretization. This yields replacing the first sum in \eqref{3LinForm} 
by a sum over 
the prismatic elements $K\in{\mathcal C}_h$ of the form 
$\sum_{K\in {\mathcal C}_h}$ and thus have 
the terms with  $\sum_{m=0}^{M-1} (\cdot,\cdot)_m$ replaced by 
$\sum_{K\in {\mathcal C}_h}(\cdot,\cdot)_K$. Hence  
\begin{multline*}
\B(G;f,g)
=
\sumKV\Big(\d_t f+G\cdot\nabla f,
g+\delta_K(\d_t g+G(f^{h,i-1})\cdot\nabla g)\Big)_K\\
+\sum_{m=1}^{M-1}\sprod{[f]}{g_+}_m +\sprod{f_+}{g_+}_0.
\end{multline*}
Therefore, the adequate norm to derive 
stability and convergence estimates for the Vlasov equation 
will be the following triple norm:
\[
\Vnorm{g}^2 =\frac{1}{2}\Big(|g_+|^2_0+|g_-|^2_M+\sum_{m=1}^{M-1}|[g]|^2_m
+2\!\sumKV\delta_K\|\d_t g+G(f^{h,i-1})\cdot\nabla g\|^2_K\Big).
\]

\section{Stability and Convergence of hp-SDFEM}

\subsection{Maxwell Equations}

\begin{lemma} [M-coercivity]\label{norm}
The bilinear form $\Bt (\cdot,\, \cdot)$ satisfies the coercivity relation
\[
\Bt (g, g) = \Mnorm{g}^2, \qquad \forall g\in \mathcal{\tilde{H}}_0.
\]
\end{lemma}
\begin{proof}
By definition of $ \Bt $ we have that
\begin{multline*}
\Bt (g,g) = \sum_{m=0}^{M-1}(\d_t g +\sum_{l=1}^3 M_l\d_l g, g)_m\\
+ \sumK\delta_K \| \d_t g + \sum_{l=1}^3 M_l\d_l g \|_K^2
+ \sum_{m=1}^{M-1} \sprod{ [ g ]}{g_+}_m + | g_+ |^2_0.
\end{multline*}
Integrating by parts we get that
\[
\sum_{m=0}^{M-1} (\d_t g, g) + \sum_{m=1}^{M-1} \sprod{[g]}{g_+}_m + |g_+|^2_0
= \frac{1}{2} \Big( \sum_{m=1}^{M-1} | [ g ] |^2_m 
+ | g_- |^2_M + | g_+ |^2_0 \Big)
\]
and since $g(t,x)=0$ on $I\times\d\Omega_x$, we have that 
\begin{equation}\label{lemma4_1}
\sum_{m=0}^{M-1} ( M_l\d_l g, g)_m = 0, \quad\mbox{for } l=1,2,3.
\end{equation}
Then, the proof follows immediately through adding all above terms.
\end{proof}

\begin{lemma} [Poincar\'e-type M-estimate]\label{konlem}
For any positive constant $ C $ we have that for 
$ g \in \mathcal{\tilde{H}}_0 $,
\[
\|g\|^2_m \leq \Big(|g_-|^2_{m+1}
+ \frac{1}{C} \|\d_t g + \sum_{l=1}^3 M_l\d_l g\|^2_m \Big) he^{2Ch} .
\]
\end{lemma}
\begin{proof}
For $ t_m < t < t_{m+1} $, we may write
\begin{equation*}
\begin{split}
\|g(t)\|^2_{\Omega_x} &= |g_-|^2_{m+1}-\int_{t}^{t_{m+1}}\frac{d}{ds}
\|g(s)\|^2_{\Omega_x}ds\\
&=|g_-|^2_{m+1}-2\int_{t}^{t_{m+1}}
(\d_s g + \sum_{l=1}^3 M_l\d_l g, g)_{\Omega_x}ds\\
& \leq |g_-|^2_{m+1}+\frac{1}{C}\|\d_t g+\sum_{l=1}^3 M_l\d_l g\|^2_m
+C\int_t^{t_{m+1}}\|g\|^2_{\Omega_x}, \\
\end{split}
\end{equation*}
where in the second equality we used \eqref{lemma4_1}. 
Finally, by Gr\"{o}nwall's lemma we have that
\[
\|g(t)\|^2_{\Omega_x}\leq\Big(|g_-|^2_{m+1}
+\frac{1}{C}\|\d_t g+\sum_{l=1}^3 M_l\d_l g\|^2_m \Big) e^{2Ch}.
\]
Now, integrating over $[t_m,t_{m+1}]$ we obtain the desired result.
\end{proof}

We proceed to the error analysis.
First, let $\tilde{W}$ be an interpolant of $W$ in the finite dimensional
discrete function space $\tilde{V}_h$ and denote by $W^{h,i}$ 
a solution to \eqref{sdm2}.
Then we represent the error as the following split
\[
\et= W- W^{h,i} = (W-\tilde{W})-(W^{h,i}-\tilde{W}) = \etat- \xit,
\]
where $\etat=(\etat_1,\ldots,\etat_6)^T$ is the interpolation error 
and $\xit=(\xit_1,\ldots,\xit_6)^T$.

To estimate the convergence rate for both the Maxwell's and the Vlasov-Maxwell
equations we use the Theorem 3.2 from \cite{AsadSopas}, which are based on 
classical interpolation estimates by \cite{Ciarlet}. 
As a consequence of this Theorem we have the following bounds
for the interpolation error $\eta=f-\tilde f$ of a function 
$f\in H^{k+1}( [0,T] \times \Omega)$
(where $\Omega$ stands for $\Omega_x$ in the Maxwell's equations
and $\Omega_x\times\Omega_v$ for the Vlasov case)
 and its gradient:
\begin{align}\label{etabound1AS}
\|\eta\|^2&\leq C\sum_K\left(\frac{h_K}{2}\right)^{2s_K+2}
\Phi_1(p_K,s_K)\|f\|^2_{s_K+1,K},\\
\label{etabound2AS}\|\mathcal{D}\eta\|^2&
\leq C\sum_K\left(\frac{h_K}{2}\right)^{2s_K}
\Phi_2(p_K,s_K)\|f\|^2_{s_K+1,K},
\end{align}
where the sums are taken over all space-time 
elements of the triangulation of the domain, $ [0, T] \times \Omega$, 
$0\leq s_K\leq\min(p_K,k)$, with $p_K$ being the local spectral order.
Closed formulas for $ \Phi_1 $ and $ \Phi_2 $ are given in Theorem 3.2 of
\cite{AsadSopas}. A less involved formula for $ \Phi_1 $ can be found in
\cite{Houston_Schwab_Suli:2002}.


Now we state the following convergence theorem for the Maxwell's equations.

\begin{theorem} \label{maxbigthe}
Assume that $W\in H^{k+1}([0,T] \times \Omega)$. Moreover
on each $K$, the parameter $\delta_K$ satisfies 
$\delta_K=C_1\frac{h_K}{p_K}$
for some constant $C_1>0$ with $p_Kh_K\leq C_2<1$ for some constant $C_2>0$.
Then there exists a constant $C>0$ independent of $p_K$, $h_K$ and $s_K$ such that
\[ 
\Mnorm{W-W^{h,i}}^2 \leq
C\!\!\sumK h_K^{2s_K+1}p_K^{-1}\Phi_{\mathrm{M}}(p_K,s_K)\|W\|^2_{s_K+1,K}\\
+ C \| f-f^{h,i-1} \|^2_{Q_T},
\]
where $\Phi_{\mathrm{M}} = \max(\Phi_1, \Phi_2) $ 
with $\mathcal{N}=\textnormal{dim } \Omega_x +1$ for $ \Phi_1 $ and $ \Phi_2 $
(recall that \textnormal{M}, as a subscript in the triple norm above, 
is to emphasis that it concerns the triple norm of the Maxwell's equations). 
\end{theorem}
\begin{proof}
We have by Lemma \ref{norm} and \eqref{ortomax} that
\[
\Mnorm{\xit}^2 = \Bt (\xit, \xit ) = \Bt(\etat,\xit)
- \Lt(b;\xit)+\Lt(b^{h,i-1};\xit).
\]
We start with the term
\[
\Bt(\etat,\xit)=\sumK \Big( \d_t\etat+\sum_{l=1}^3 M_l\d_l\etat,
\xit+\delta_K(\d_t\xit+\sum_{l=1}^3 M_l\d_l\xit) \Big)_K
+ \sum_{m=1}^{M-1} \sprod{[ \etat ]}{\xit_+}_m + \sprod{\etat_+}{\xit_+}_0.
\]
Partial integration gives the identities
\[
(\d_t\etat, \xit )_m = \int_{\Omega_x} [\, \etat \xit \, ]_{t=t_m}^{t_{m+1}} dx
- ( \etat , \d_t\xit )_m
= \sprod{\etat_-}{\xit_-}_{m+1} - \sprod{\etat_+}{\xit_+}_m - ( \etat, \d_t\xit )_m
\]
and
\[
(M_l\d_l\etat, \xit)_m = -(\etat, M_l\d_l\xit)_m ,
\]
since $ \etat $ and $ \xit $ have compact support in $ \Omega_x $.
Inserting these equations into the expression for $\Bt(\etat,\xit)$
we end up with the following equality
\begin{multline*}
|\Bt(\etat,\xit)| = |\sprod{\etat_-}{\xit_-}_M-\sum_{m=1}^{M-1}\sprod{\etat_-}{[\xit]}_m\\
-\sum_{m=0}^{M-1}(\etat,\d_t\xit+\sum_{l=1}^3 M_l\d_l\xit)_m
+\!\sumK\delta_K \Big( \d_t\etat+\sum_{l=1}^3 M_l\d_l\etat,\d_t\xit+\sum_{l=1}^3 M_l\d_l\xit \Big)_K| .
\end{multline*}
Further, using some standard inequalities it follows that
\[
|\Bt(\etat,\xit)| \leq \frac{1}{16}\Mnorm{\xit}^2
+32\sum_{m=0}^{M-1}|\etat_-|^2_{m+1}+\!\sumK\Big(\frac{32}{\delta_K}\|\etat\|^2_K
+32\delta_K\|\d_t\etat+\sum_{l=1}^3 M_l\d_l\etat\|^2_K\Big).
\]
Now let us estimate the second term
\begin{multline*}
|\Lt(b^{h,i-1};\xit)-\Lt(b;\xit) | = \Big|\sumK\Big(b^{h,i-1}-b,
\xit+\delta_K(\d_t\xit+\sum_{l=1}^3 M_l\d_l\xit)\Big)_K\Big|\\
\leq \frac{1}{16} \Mnorm{\xit}^2+\sumK(\frac{C}{2}+8\delta_K)\|b-b^{h,i-1}\|^2_K
+\sum_{m=0}^{M-1}\big(\frac{1}{2C}\|\etat\|^2_m+\frac{1}{2C}\|\et\|_m^2\big).
\end{multline*}
The above two inequalities and Lemma 
\ref{konlem}, with properly chosen $C$
and the bound on $p_K h_K$, imply the estimate
\begin{multline*}
\Mnorm{\xit}^2 \leq \frac{1}{8} \Mnorm{\xit}^2 + \frac{7}{16} \Mnorm{\et}^2
+C\|f-f^{h,i-1}\|^2_{Q_T}\\
+C\sum_{m=0}^{M-1}(|\etat_-|^2_{m+1}+h|\et_-|^2_{m+1})
+C\!\sumK\Big((1+\frac{1}{\delta_K})\|\etat\|^2_K
+\delta_K\|\d_t\etat+\sum_{l=1}^3 M_l\d_l\etat\|^2_K\Big).
\end{multline*}
Hiding the $\xit$-term on the right hand side in the $\xit$-term on the left hand side,
gives us the following inequality
\begin{multline*}
\Mnorm{\xit}^2 \leq \frac{1}{2} \Mnorm{\et}^2+ C \| f-f^{h,i-1} \|^2_{Q_T}
+C\sum_{m=0}^{M-1}(|\etat_-|^2_{m+1}+h|\et_-|^2_{m+1})\\
+C\!\sumK\Big((1+\delta_K^{-1})\|\etat\|^2_K
+\delta_K\|\d_t\etat+\sum_{l=1}^3 M_l\d_l\etat\|^2_K\Big).
\end{multline*}
Thus, we have estimated $\Mnorm{\xit}^2$. This implies that
\begin{multline}\label{est_e1}
\Mnorm{\et}^2 \leq \Mnorm{\etat}^2 + \Mnorm{\xit}^2
\leq \frac{1}{2} \Mnorm{\et}^2+ C\|f-f^{h,i-1}\|^2_{Q_T} 
+ C\sum_{m=1}^M h |\et_-|^2_m\\
+C\Big(|\etat_+|^2_0+\sum_{m=0}^{M-1}|\etat_-|^2_{m+1}
+\!\sumK(\delta_K^{-1}\|\etat\|_K^2
+\delta_K\|\d_t\etat+\sum_{l=1}^3M_l\d_l\etat\|^2_K)
+\sum_{m=1}^{M-1}|[\etat]|^2_m \Big).
\end{multline}

Now we have to estimate the interpolation error terms:
\[
\begin{split}
J_1 &:= \sumK\Big(\delta_K^{-1}\|\etat\|_K^2
+\delta_K\|\d_t\etat+\sum_{l=1}^3M_l\d_l\etat\|^2_K\Big),\\
J_2 &:=|\etat_+|^2_0+\sum_{m=0}^{M-1}|\etat_-|^2_{m+1}
+\sum_{m=1}^{M-1}|[\etat]|^2_m.\\
\end{split}
\]
For the term $J_1$ we use \eqref{etabound1AS} and \eqref{etabound2AS} to get
\begin{equation}\label{estT1}
J_1\leq C\sumK\left(\frac{h_K}{2}\right)^{2s_K}\Phi_\mathrm{M}(p_K,s_K)
p_K^{-2}\big(\delta_K^{-1}h^2_K+\delta_K\big)\|W\|_{s_K+1,K}^2.
\end{equation}
To estimate the term $J_2$ we use the trace 
estimate combined with the inverse inequality and get 
\begin{equation}\label{est_eta_dK}
\|\etat\|_{\d K}^2 \leq C\Big(\big(\sum_{l=1}^6\|\nabla\etat_l\|_K\big)
\|\etat\|_K
+h_K^{-1}\|\etat\|_K^2\Big),
\end{equation}
to obtain
\begin{multline}\label{estT2}
J_2 \leq C\sumK\left[\left(\frac{h_K}{2}\right)^{s_K}
\Phi_\mathrm{M}^{1/2}(p_K,s_K)
\left(\frac{h_K}{2}\right)^{s_K+1}p_K^{-1}\Phi_\mathrm{M}^{1/2}(p_K,s_K)
\right.\\
+\left.h_K^{-1}\left(\frac{h_K}{2}\right)^{2s_K+2}
p_K^{-2}\Phi_\mathrm{M}(p_K,s_K)\right]
\|W\|_{s_k+1,K}^2,
\end{multline}
where for all terms in $J_2$ we combined \eqref{est_eta_dK}
with \eqref{etabound1AS} and \eqref{etabound2AS}.

Now, moving the triple norm of $\et$ on the right hand side of \eqref{est_e1}
to the left hand side and using estimates \eqref{estT1} and \eqref{estT2}
it follows that
\begin{multline} \label{ineq}
\Mnorm{\et}^2 \leq
C\sumK h_K^{2s_K+1}p_K^{-1}\Phi_\mathrm{M}(p_K,s_K)\|W\|^2_{s_K+1,K}\\
+ C \| f-f^{h,i-1} \|^2_{Q_T} + C \sum_{m=1}^M h | \et_- |^2_m .
\end{multline}

The next step is to apply a discrete Gr\"onwall lemma of the form:
suppose $ \{ a_\ell \}^M_1 $ satisfies
\[
a_\ell \leq C_1 + C_2 \sum_{j=1}^\ell a_j h,\qquad \ell= 1, \ldots , M.
\]
If $ h \leq 1/2 C_2 $, then we have that 
\[
a_\ell \leq 2 C_1 e^{2C_2 (\ell-1)h} \quad \mbox{for} 
\quad  \ell=1, \ldots , M .
\]
The discrete Gr\"onwall's lemma yields
\[
|\et_-|^2_\ell \leq
C\sumK h_K^{2s_K+1}p_K^{-1}\Phi_\mathrm{M}(p_K,s_K)\|W\|^2_{s_K+1,K}\\
+ C \| f-f^{h,i-1} \|^2_{Q_T}
\]
for $ \ell = 1, \ldots , M $. Plugging these inequalities into \eqref{ineq}
will give the stated error estimate and the proof is complete. 
\end{proof}
\begin{remark}
Theorem \ref{maxbigthe} and Lemma \ref{konlem}, with properly chosen $C$
and the bound on $p_Kh_K$, with the definition of $\Mnorm{\cdot}$
imply the following $L_2$-norm error estimate
\begin{multline} \label{maxl2}
\|W-W^{h,i}\|^2_{I\times\Omega_x} \leq 
C\!\!\sumK h_K^{2s_K+1}p_K^{-1}\Phi_\mathrm{M}(p_K,s_K)\|W\|^2_{s_K+1,K}\\
+ Cph\|f-f^{h,i-1}\|^2_{Q_T}.
\end{multline}
\end{remark}

\subsection{Vlasov-Maxwell Equations}

\begin{lemma} [V-coercivity] \label{Vstab}
We have that
\[
\B ( G(f^{h,i-1} ); g, g) = \Vnorm{g}^2 \quad\forall g \in \mathcal{H}_0.
\]
\end{lemma}
\begin{proof}
Taking into account that $\nabla G(f^{h,i-1}) = 0$, $g$ is zero on 
$\d\Omega$ and
following the proof of Lemma \ref{norm} we get the desired result.
\end{proof}
\begin{lemma} [Poincar\'e-type V-estimate] \label{V2normest}
For any constant $ C $ we have for $ g \in \mathcal{H}_0 $,
\[
\|g\|^2_m \leq\Big( |g_-|^2_{m+1}
+\frac{1}{C}\|\d_t g+G(f^{h,i-1})\cdot\nabla g\|^2_m
 \Big) he^{Ch} .
\]
\end{lemma}
The proof is similar to that of Lemma \ref{konlem} and therefore is 
omitted.

Now we proceed with the error analysis.
First we let $ \tilde{f} $ be an interpolant of $ f $.
Then we set
\[
e= f-f^{h,i} = (f-\tilde{f}) - (f^{h,i} - \tilde{f}) = \eta- \xi.
\]
We state the following convergence theorem.
\begin{theorem} \label{bigthevla}
Let $f^{h,i}$ be a solution to \eqref{sdv2} and assume that the exact solution $f$
of \eqref{vlasov} is in the Sobolev class $ H^{k+1}(Q_T)$ and satisfies the bound
\begin{equation}\label{Vinfnorm}
\| \nabla f \|_\infty +  \| G(f) \|_\infty + \| \nabla \eta \|_\infty \leq C,
\end{equation}
and the parameter $\delta_K$ on each $K$ satisfies $\delta_K=C_1\frac{h_K}{p_K}$
for some positive constant $C_1$ with $p_Kh_K\leq C_2<1$ for some constant $C_2>0$.
Then there exists a constant $C>0$ independent of $p_K$, $s_K$ and $h_K$ such that
\begin{multline}\label{errest1}
\Vnorm{f-f^{h,i}}^2\leq 
C\bigg(\sumK h_K^{2s_K+1}p_K^{-1}\Phi_\mathrm{M}(p_K,s_K)\|W\|^2_{s_K+1,K}
+ph\|f-f^{h,i-1}\|^2_{Q_T}\\
+\!\!\sumKV h_K^{2s_K+1}p_K^{-1}\Phi_\mathrm{V}(p_K,s_K)\|f\|^2_{s_K+1,K}\bigg),
\end{multline}
where $0\leq s_K\leq\min(p_K,k)$ and the subscript $V$ in the triple norm above, 
as well as a subscript for $\Phi$ is to emphasize that
these quantities are in the Vlasov part. 
Here, $ \Phi_\mathrm{V} = \max(\Phi_1, \Phi_2) $ with
$ \mathcal{N} = \dim \Omega_x + \dim \Omega_v + 1 $ for $ \Phi_1 $ and $ \Phi_2 $.
\end{theorem}
\begin{proof}
By \eqref{sdv2} and Lemma \ref{Vstab} we get that
\[
\Vnorm{\xi}^2=\B(G(f^{h,i-1});\xi,\xi)=\L(\xi)-\B(G(f^{h,i-1});\tilde{f},\xi)=T_1+T_2,
\]
where
\[
T_1 = \B (G(f^{h,i-1}); \eta, \xi)
\]
and
\[
T_2 = \B (G(f); f, \xi) - \B (G(f^{h,i-1}); f, \xi).
\]
We start with the term $T_1$.
Integrating by parts and using the facts that $\eta$ and $\xi$ are zero on $\d\Omega$
and $\nabla G(f^{h,i-1})=0$, we get
\[
\begin{split}
T_1&=\sumKV
\big(\d_t\eta+G(f^{h,i-1})\cdot\nabla\eta,\xi+\delta_K(\d_t\xi+G(f^{h,i-1})\cdot\nabla\xi)\big)_K\\
&\qquad+\sum_{m=1}^{M-1}\sprod{[ \eta ]}{\xi_+}_m + \sprod{\eta_+}{\xi_+}_0\\
&=-(\eta,\d_t\xi+G(f^{h,i-1})\cdot\nabla\xi)_{Q_T}+\sprod{\eta_-}{\xi_-}_M
-\sum_{m=1}^{M-1}\sprod{\eta_-}{[\xi]}_m\\
&\qquad
+\sumKV\delta_K\big(\d_t\eta+G(f^{h,i-1})\cdot\nabla\eta,\d_t\xi+G(f^{h,i-1})\cdot\nabla\xi\big)_K.
\end{split}
\]
Now using Cauchy-Schwarz inequality we obtain the estimate
\begin{equation*}
|T_1|\leq\frac{1}{8}\Vnorm{\xi}^2+C\Big(\sum_{m=1}^M|\eta_-|_m^2+\delta_K^{-1}\|\eta\|_K^2+
\sumKV\big(\delta_K\|\d_t\eta+G(f^{h,i-1})\cdot\nabla\eta\|_K^2\big)\Big),
\end{equation*}
where for the last term we have the bound
\begin{multline}\label{eta_bound}
\|\d_t\eta+G(f^{h,i-1})\cdot\nabla\eta\|_K\leq\\
\leq\|\d_t\eta\|_K+\|G(f)\|_\infty\|\nabla\eta\|_K+\|\nabla\eta\|_\infty\|G(f^{h,i-1})-G(f)\|_K.
\end{multline}
Next we estimate $T_2$:
\begin{equation*}
\begin{split}
|T_2|&\leq\sumKV\delta_K\big|\big((G(f)-G(f^{h,i-1}))\cdot\nabla f,
\d_t\xi+G(f^{h,i-1})\cdot\nabla\xi\big)_K\big|\\
&\qquad+\big|\big((G(f)-G(f^{h,i-1}))\cdot\nabla f,\xi\big)_{Q_T}\big|\\
&\leq C\Big(\delta_K\|G(f)-G(f^{h,i-1})\|_{Q_T}^2\|\nabla f\|_\infty^2
+\!\!\sumKV\frac{\delta_K}{8}\|\d_t\xi+G(f^{h,i-1})\cdot\nabla\xi\|_K^2\Big)\\
&\qquad+
C\|G(f)-G(f^{h,i-1})\|_{Q_T}\|\nabla f\|_\infty\|\xi\|_{Q_T}.
\end{split}
\end{equation*}
To proceed we need to estimate $\|G(f)-G(f^{h,i-1})\|_{Q_T}$.
By the definition of $G$ we have that
\[
G(f)-G(f^{h,i-1})=(0, E-E^{h,i}+\hat{v}\times(B-B^{h,i})),
\]
which gives
\[
\|G(f)-G(f^{h,i-1})\|_{Q_T}^2\leq C(\|E-E^{h,i}\|_{Q_T}^2+\|\hat{v}\times(B-B^{h,i})\|_{Q_T}^2).
\]
Hence using \eqref{maxl2} we obtain
\begin{equation}\label{G_bound}
\|G(f)-G(f^{h,i-1})\|_{Q_T}^2\leq 
C\!\!\sumK h_K^{2s_K+1}\Psi_\mathrm{M}(p_K,s_K)+Cph\|f-f^{h,i-1}\|^2_{Q_T},
\end{equation}
where we denote
\[
\Psi_\mathrm{M}(p_K,s_K):=p_K^{-1}\Phi_\mathrm{M}(p_K,s_K)\|W\|^2_{s_K+1,K}.
\]
Now combining the estimates for $T_1$ and $T_2$ together with \eqref{eta_bound},
\eqref{G_bound} and \eqref{Vinfnorm} we have
\begin{multline*}
\Vnorm{\xi}^2\leq\frac{1}{4}\Vnorm{\xi}^2+C\|\xi\|_{Q_T}^2
+C\!\!\sumK h_K^{2s_K+1}\Psi_\mathrm{M}(p_K,s_K)+Cph\|f-f^{h,i-1}\|^2_{Q_T}\\
+C\!\sum_{m=1}^M|\eta_-|^2_m+C\!\!\sumK\Big(\delta_K^{-1}\|\eta\|^2_K
+\delta_K\big(\|\d_t\eta\|_K^2+\|\nabla\eta\|_K^2\big)\Big).
\end{multline*}
Moving the triple norm to the left hand side
and estimating $\|\xi\|^2_{Q_T}\leq C(\|e\|^2_{Q_T}+\|\eta\|^2_{Q_T})$,
will give us the following inequality
\begin{multline*}
\Vnorm{\xi}^2\leq C\|e\|_{Q_T}^2+C\|\eta\|_{Q_T}^2
+C\!\!\sumK h_K^{2s_K+1}\Psi_\mathrm{M}(p_K,s_K)+Cph\|f-f^{h,i-1}\|^2_{Q_T}\\
+C\!\sum_{m=1}^M|\eta_-|^2_m+C\!\!\sumK\Big(\delta_K^{-1}\|\eta\|^2_K
+\delta_K\big(\|\d_t\eta\|_K^2+\|\nabla\eta\|_K^2\big)\Big).
\end{multline*}
We now estimate $\Vnorm{e}$ as follows
\begin{multline*}
\Vnorm{e}^2\leq 2\Vnorm{\xi}^2+2\Vnorm{\eta}^2\leq\\
\leq C\|e\|_{Q_T}^2+C\|\eta\|_{Q_T}^2
+C\!\!\sumK h_K^{2s_K+1}\Psi_\mathrm{M}(p_K,s_K)+Cph\|f-f^{h,i-1}\|^2_{Q_T}\\
+C\!\sum_{m=1}^M|\eta_-|^2_m+C\!\!\sumK\Big(\delta_K^{-1}\|\eta\|^2_K
+\delta_K\big(\|\d_t\eta\|_K^2+\|\nabla\eta\|_K^2\big)\Big)
+|\eta_+|^2_0+\sum_{m=1}^{M-1}|[\eta]|^2_m.
\end{multline*}
Using Lemma \ref{V2normest} for the term $\|e\|^2_{Q_T}$
with an appropriately chosen constant $C$ and the bound on $p_Kh_K$ we get
\begin{multline*}
\Vnorm{e}^2\leq \frac{1}{2}\Vnorm{e}^2+Ch\!\sum_{m=1}^M|e_-|^2_m
+C\!\!\sumK h_K^{2s_K+1}\Psi_\mathrm{M}(p_K,s_K)+Cph\|f-f^{h,i-1}\|^2_{Q_T}\\
+C\|\eta\|_{Q_T}^2+C\!\sum_{m=1}^M|\eta_-|^2_m+|\eta_+|^2_0+\sum_{m=1}^{M-1}|[\eta]|^2_m\\
+C\!\!\sumK\Big(\delta_K^{-1}\|\eta\|^2_K
+\delta_K\big(\|\d_t\eta\|_K^2+\|\nabla\eta\|_K^2\big)\Big).
\end{multline*}
Recalling \eqref{etabound1AS} and \eqref{etabound2AS} for all 
$\eta$-terms we obtain
(in a similar way as in the proof of Theorem \ref{maxbigthe} 
for the terms $J_1$ and $J_2$)
\begin{multline*}
\Vnorm{e}^2\leq Ch\!\sum_{m=1}^M|e_-|^2_m
+C\!\!\sumK h_K^{2s_K+1}\Psi_\mathrm{M}(p_K,s_K)+Cph\|f-f^{h,i-1}\|^2_{Q_T}\\
+C\!\!\sumKV h_K^{2s_K+1}p_K^{-1}\Phi_\mathrm{V}(p_K,s_K)\|f\|^2_{s_K+1,K}.
\end{multline*}
Now we use the discrete Gr\"{o}nwall lemma stated in the proof of 
Theorem \ref{maxbigthe}
and we get the following inequality
\begin{multline*}
\Vnorm{e}^2\leq 
C\!\!\sumK h_K^{2s_K+1}p_K^{-1}\Phi_\mathrm{M}(p_K,s_K)\|W\|^2_{s_K+1,K}
+Cph\|f-f^{h,i-1}\|^2_{Q_T}\\
+C\!\!\sumKV h_K^{2s_K+1}p_K^{-1}\Phi_\mathrm{V}(p_K,s_K)\|f\|^2_{s_K+1,K},
\end{multline*}
which gives the desired result and completes the proof. 
\end{proof}

\begin{corollary}
Under the assumptions of Theorem \ref{bigthevla} we have
\begin{multline*}
\Vnorm{f-f^{h,i}}^2\leq 
C\!\!\sumK h_K^{2s_K+1}p_K^{-1}\Phi_\mathrm{M}(p_K,s_K)\|W\|^2_{s_K+1,K}
+C\big(ph\big)^i\\
+C\!\!\sumKV h_K^{2s_K+1}p_K^{-1}\Phi_\mathrm{V}(p_K,s_K)\|f\|^2_{s_K+1,K}.
\end{multline*}
\end{corollary}
\begin{proof}
Using Lemma \ref{V2normest} with a properly chosen constant $C$
and the bound on $p_Kh_K$ we get
\[
\|f-f^{h,i}\|^2_{s_K+1,K}\leq S+Cph\|f-f^{h,i-1}\|^2_{s_K+1,K},
\]
where $S$ denotes the terms with sums from 
the right hand side of \eqref{errest1}.
Using this inequality repeatedly will give us
\[
\|f-f^{h,i}\|^2_{s_K+1,K}\leq CS+C(ph)^i,
\]
which ends the proof.
\end{proof}

\begin{remark}[The DG approach]
The whole theory developed in the previous sections 
work for the streamline diffusion based discontinuous Galerkin (SDDG) 
method as well. 
In this method our approximations are also allowed to have jump 
discontinuities  
across the inter-element boundaries. Then the norms, including 
the sum over such jump terms, are much more involved. 
However, as we mentioned above, the analysis although lengthy 
follow the same path. 
\end{remark}

\section{Nitsche's method for Maxwell equations}

Alternative, yet more desirable numerical scheme for
the Maxwell's equations can be obtained using a symmetrizing
penalty approach known as Nitsche's method.
This, however cannot be extended to the Vlasov part due 
to the hyperbolic nature of the Vlasov equation.
Nevertheless, the advantages of the symmetrizing are
overwhelming.
Therefore below we include analysis of the Nitsche's 
approach for the Maxwell part.

Recall that we have the Maxwell's equations given by
\begin{equation*} 
\begin{aligned}
& E_{t} - \nabla_x \times B = - j, \\
& B_{t} + \nabla_x \times E = 0.
\end{aligned}
\end{equation*}
We differentiate the first equation with respect to time to get
\begin{equation}\label{MaxwellA}
E_{tt} - \nabla_x \times B_t = -j_t,
\end{equation}
and plug the second Maxwell's equation into the equation \eqref{MaxwellA} 
to obtain 
\begin{equation}\label{MaxwellB}
E_{tt} + \nabla_x \times (\nabla_x \times E ) = -j_t.
\end{equation}
Multiplying \eqref{MaxwellB} with 
$ g \in H(\textnormal{curl}, \Omega_x) = 
\{ v \in L_2(\Omega_x): \curl v \in L_2(\Omega_x ) \} $
and integrating over $ \Omega_x $ yields 
\begin{equation} \label{NM1}
\int_{\Omega_x} E_{tt} \cdot g \, dx +
\int_{\Omega_x} \gradx \times 
(\gradx \times E) \cdot g \, dx = -\int_{\Omega_x} j_t \cdot g \, dx.
\end{equation}
Now recall the Green's formula
\[
\int_{\Omega_x} u \cdot \curl v \, dx - 
\int_{\Omega_x} \curl u \cdot v \, dx = 
- \int_{\Gamma_x} (u \times n) \cdot v \, ds,
\]
where $ n=n(x) $ is the unit outward normal to the boundary at the point 
$x\in \partial\Omega_x$. 
Apply Green's formula to \eqref{NM1}
\[
\int_{\Omega_x} E_{tt} \cdot g \, dx + 
\int_{\Omega_x} \curl E \cdot \curl g \, dx
- \int_{\Gamma_x} \curl E \cdot (g \times n) \, ds = 
-\int_{\Omega_x} j_t \cdot g \, dx.
\]
This relation being non-symmetric 
(see the contribution from the boundary terms) 
causes severe restrictions in, e.g., deriving stability estimates. 
To circumvent such draw-backs Nitsche introduced a symmetrized 
scheme for elliptic and parabolic problems (see \cite{Nitsche}), which is also known as 
the {\sl penalty method}. This can be seen in, e.g., \cite{Burman_Hansbo} and \cite{Sticko_Kreiss}. 
In our case Nitsche's method is performed 
by the add of extra boundary terms making the bilinear form
symmetric and coercive, viz. 
\begin{equation*}
\begin{aligned}
\int_{\Omega_x} E_{tt} \cdot g \, dx 
+ \int_{\Omega_x} \curl E \cdot \curl g \, dx
- \int_{\Gamma_x} \curl E \cdot (g \times n) \, ds \\
- \int_{\Gamma_x} (E \times n) \cdot \curl g \, ds 
+ \frac{\gamma}{h} \int_{\Gamma_x} E \cdot g \, ds
= -\int_{\Omega_x} j_t \cdot g \, dx.
\end{aligned}
\end{equation*}
Here $ \gamma $ is a constant that will be specified later.
Now, we define the symmetric bilinear form
\[
\begin{aligned}
a(E, g) :=& \int_{\Omega_x} \curl E \cdot \curl g \, dx- 
\int_{\Gamma_x} \curl E \cdot (g \times n) \, ds \\
&- \int_{\Gamma_x} (E \times n) \cdot \curl g \, ds 
+ \frac{\gamma}{h} \int_{\Gamma_x} E \cdot g \, ds
\end{aligned}
\]
and the element space of piecewise linear polynomials
\[
V^x_h=\{g\in H(\textnormal{curl}, \Omega_x): g|_{\tau_x} \in P_1(\tau_x ),
\, \forall \tau_x\in T^x_h\}.
\]
Thus, we can formulate the semi-discrete problem as: for each fixed $ t $,
find $ E^h (t, \cdot) \in V^x_h $, such that
\begin{equation} \label{NM2}
(E_{tt}^h, g)_{\Omega_x} + a(E^h, g) = 
-(j_t, g)_{\Omega_x} \quad \forall g \in V^x_h.
\end{equation}

It is straightforward to observe the consistency of the method.
\begin{lemma} [Consistency]\label{consisLemma}
The exact solution $E$ of \eqref{MaxwellB} satisfies
\[
(E_{tt}, g)_{\Omega_x} + a(E, g) = 
-(j_t, g)_{\Omega_x} \quad \forall g \in H(\textnormal{curl}, \Omega_x).
\]
\end{lemma}

Now, defining the mesh dependent discrete norm
\[
\norm{g}^2_h := \norm{\curl g}^2_{\Omega_x} + \norm{h^{-1/2} g}^2_{\Gamma_x},
\]
we have
\[
\begin{aligned}
a(g, g) &= \norm{\curl g}^2_{\Omega_x} - 2(\curl g, g \times n)_{\Gamma_x}
+ \frac{\gamma}{h}\norm{g}^2_{\Gamma_x} \\
&\geq \norm{\curl g}^2_{\Omega_x} - 
2\norm{h^{1/2} \curl g}_{\Gamma_x} \norm{h^{-1/2} g \times n}_{\Gamma_x}
+ \gamma \norm{h^{-1/2} g}^2_{\Gamma_x} \\
&\geq \norm{\curl g}^2_{\Omega_x} - 
\frac{1}{\alpha}\norm{h^{1/2} \curl g}^2_{\Gamma_x}
- \alpha \norm{h^{-1/2} g \times n}^2_{\Gamma_x} 
+ \gamma \norm{h^{-1/2} g}^2_{\Gamma_x} \\
&\geq \frac{\alpha - \tilde{C}}{\alpha} \norm{\curl g}^2_{\Gamma_x} + 
(\gamma - 4\alpha) \norm{h^{-1/2} g}^2_{\Gamma_x},
\end{aligned}
\]
where in the last inequality we used the following trace estimate: 
\[
\norm{h^{1/2} \curl g}^2_{\Gamma_x} 
\leq \tilde{C} \norm{\curl g}^2_{\Omega_x} \quad \forall g \in V^x_h
\]
and the trivial inequality 
\begin{equation}\label{triv_ineq}
\norm{h^{-1/2} g \times n}^2_{\Gamma_x} 
\leq 4 \norm{h^{-1/2} g}^2_{\Gamma_x} \quad \forall g \in L_2 (\Gamma_x).
\end{equation}
Now, if we choose the constants $ \gamma $ and 
$ \alpha $, such that $ \gamma > 4 \alpha $ and $ \alpha > \tilde{C} $,
then we have proved the following coercivity result.
\begin{lemma} [Coercivity] \label{coercLemma}
If $ \gamma $ is large enough, then there exists a constant $ C > 0 $,
 such that
\[
a(g, g) \geq C \norm{g}^2_h \quad \forall g \in V^x_h.
\]
\end{lemma}

To have continuity of the form $a(\cdot,\cdot)$
we need to define a mesh dependent triple norm as 
\[
\Tnorm{g}^2_h := \norm{g}^2_h + \norm{h^{1/2} \curl g}^2_{\Gamma_x}.
\]
Then we have the following lemma.
\begin{lemma} [Continuity] \label{contLemma}
The bilinear form $a(\cdot,\cdot)$ is continuous
with respect to the triple norm $\Tnorm{\cdot}_h$ and we have the following estimate
\[
|a(u, v)|\leq (9 + \gamma) \Tnorm{u}_h \Tnorm{v}_h. 
\]
\end{lemma}
\begin{proof}
Using inequality \eqref{triv_ineq} and simple algebra we get the result:
\[
\begin{aligned}
|a(u, v)| \leq& \norm{\curl u}_{\Omega_x} \norm{\curl v}_{\Omega_x}
+ \norm{h^{1/2} \curl u}_{\Gamma_x} \norm{h^{-1/2} v \times n}_{\Gamma_x} \\
& + \norm{h^{-1/2} u \times n}_{\Gamma_x} \norm{h^{1/2} \curl v}_{\Gamma_x}
+ \gamma \norm{h^{-1/2} u}_{\Gamma_x} \norm{h^{-1/2} v}_{\Gamma_x} \\
\leq& \Tnorm{u}_h \Tnorm{v}_h + 4\Tnorm{u}_h \Tnorm{v}_h
+ 4 \Tnorm{u}_h \Tnorm{v}_h + \gamma \Tnorm{u}_h \Tnorm{v}_h \\
\leq& (9 + \gamma) \Tnorm{u}_h \Tnorm{v}_h.
\end{aligned}
\]
\end{proof}
For the triple norm $ \Tnorm{\cdot}_h $ we have the following inverse estimate
\[
\Tnorm{g}_h \leq C h^{-1} \norm{g}_{\Omega_x},
\]
which holds for all $ g \in V^x_h $.
We note that the trace estimate implies the coercivity of the form $a(\cdot,\cdot)$
also in the triple norm.

\subsection{A modified Ritz projection}

Let us define a projection $ Q_h : H(\textnormal{curl}, \Omega_x) \rightarrow V^x_h $ by
\begin{equation*} 
a(Q_h u, v) = a(u, v) \quad \forall v \in V^x_h.
\end{equation*}
We have the following error estimates for the projection $ Q_h $:
\begin{lemma} \label{projlemma}
There exists a constant $ C $, such that
\begin{equation}
\norm{u - Q_h u}_{\Omega_x} + h\Tnorm{u - Q_h u}_h \leq C h^2 \norm{u}_{H^2(\Omega_x)}.
\end{equation}
\end{lemma}

\begin{proof}
Consider the stationary problem
\[
\begin{aligned}
\curl ( \curl \fii ) =& \, f \quad \mbox{in} \,\,\, \Omega_x, \\
\fii =& \, 0 \quad \mbox{on} \,\,\, \Gamma_x.
\end{aligned}
\]
The Nitsche formulation for this problem is given by:
Find $ \fii^h \in V^x_h $, such that
\begin{equation} \label{ProjB}
a(\fii^h, \chi) = (f, \chi) \quad \forall \chi \in V^x_h.
\end{equation}
Further, we have the Galerkin orthogonality
\begin{equation}\label{Galortho}
a(\fii - \fii^h, \chi) = 0 \quad \forall \chi \in V^x_h.
\end{equation}
Hence, the projection $ Q_h $ can be seen as the solution operator of 
\eqref{ProjB}.
We therefore need an a priori error estimate of \eqref{ProjB}.
To this end we split the error into two terms $ \fii - \fii^h = 
(\fii - I_h \fii) + (I_h \fii - \fii^h) = \eta + \xi $,
where $ I_h $ is the standard nodal interpolation operator.
By coercivity, continuity of $ a(\cdot, \cdot) $ and the Galerkin orthogonality 
\eqref{Galortho} we have that 
\[
\Tnorm{\xi}_h^2 \leq C a(\xi, \xi) = 
- C a(\xi, \eta) \leq C \Tnorm{\xi}_h \Tnorm{\eta}_h.
\]
It follows that $ \Tnorm{\xi}_h \leq C \Tnorm{\eta}_h $, 
so it remains to estimate $ \Tnorm{\eta}_h $.
Below we estimate each term in $ \Tnorm{\eta}_h $ separately. 
For the interpolation error we have 
\[
\norm{\curl \eta}^2_{\Omega_x} \leq 2 
\norm{\gradx \eta}^2_{\Omega_x} \leq C h^2 \norm{\fii}^2_{H^2(\Omega_x)}.
\]
As for the boundary integrals, by trace inequality we have the estimates
\[
\begin{aligned}
\norm{h^{-1/2} \eta}^2_{\Gamma_x} &\leq 
C \norm{h^{-1/2} \eta}_{\Omega_x} \norm{h^{-1/2} \eta}_{H^1(\Omega_x)} \\
&\leq C \left( h^{-1} \norm{h^{-1/2} \eta}^2_{\Omega_x} 
+ h \norm{h^{-1/2} \eta}^2_{H^1(\Omega_x)} \right) \\
&\leq C h^2 \norm{\fii}^2_{H^2(\Omega_x)}, 
\end{aligned}
\]
and similarly 
\[
\begin{aligned}
\norm{h^{1/2} \curl \eta}^2_{\Gamma_x} &\leq 
C \left( h^{-1} \norm{h^{1/2} \curl \eta}^2_{\Omega_x}
+ h \norm{h^{1/2} \gradx (\curl \eta)}^2_{\Omega_x} \right) \\
&\leq C h^2 \norm{\fii}^2_{H^2(\Omega_x)}, 
\end{aligned}
\]
where, in both estimates, in the last inequalities 
we have used the interpolation  estimates. 
Summing up we end up with 
\[
\Tnorm{\eta}_h \leq C h \norm{\fii}_{H^2(\Omega_x)}.
\]
It remains to estimate the error in the $ L_2 $-norm. To this end 
we consider the auxiliary problem
\[
\begin{aligned}
\curl ( \curl \psi ) =& \, \fii - \fii^h \quad \mbox{in} \,\,\, \Omega_x, \\
\psi =& \, 0 \quad \mbox{on} \,\,\, \Gamma_x.
\end{aligned}
\]
Multiplying the first equation with $ \fii - \fii^h $ and integrating 
over $\Omega_x$ yields 
\[
\begin{aligned}
\norm{\fii - \fii^h}^2_{\Omega_x} &= ( \fii - \fii^h, \curl (\curl \psi))_{\Omega_x} \\
&= (\curl (\fii - \fii^h), \curl \psi)_{\Omega_x} - \sprod{(\fii - \fii^h) \times n}{\curl \psi}_{\Gamma_x} \\
&= a(\fii - \fii^h, \psi) = a(\fii - \fii^h, \psi - I_h \psi) \\
&\leq \Tnorm{\fii - \fii^h}_h \Tnorm{\psi - I_h \psi}_h \\
&\leq C h^2 \norm{\fii}_{H^2(\Omega_x)} \norm{\psi}_{H^2(\Omega_x)}.
\end{aligned}
\]
By the stability of the elliptic problem 
\[
\norm{\psi}_{H^2(\Omega_x)} \leq C \norm{\fii - \fii^h}_{\Omega_x}
\]
the estimate for the $ L_2 $-norm follows.
\end{proof}

\subsection{Convergence}

Let us split the error as
\[
E - E^h = (E - Q_h E) + (Q_h E - E^h) = \rho + \theta.
\]
In order to bound $ \theta $ we note that
\[
\begin{aligned}
(\theta_{tt}, \chi)_{\Omega_x} + a(\theta, \chi) &= 
(Q_h E_{tt}, \chi)_{\Omega_x} + a(Q_h E, \chi)
 - (E^h_{tt}, \chi)_{\Omega_x} - a(E^h, \chi) \\
&= (Q_h E_{tt}, \chi)_{\Omega_x} + a(Q_h E, \chi) 
- (E_{tt}, \chi)_{\Omega_x} - a(E, \chi) \\
&= - (\rho_{tt}, \chi)_{\Omega_x}
\end{aligned}
\]
for $ \chi \in V^x_h $. We choose $ \chi = \theta_t $ to get
\[
(\theta_{tt}, \theta_t )_{\Omega_x} + a(\theta, \theta_t ) 
= - (\rho_{tt}, \theta_t)_{\Omega_x},
\]
which leads to
\[
\frac{1}{2} \frac{d}{dt} \left( \norm{\theta_t }^2_{\Omega_x} 
+ a(\theta, \theta)\right) \leq \norm{\rho_{tt}}_{\Omega_x} \norm{\theta_t}_{\Omega_x}.
\]
Integrating in time over $ [0, t] $ and noting that $ \theta (0) 
= \theta_t (0) = 0 $ we get the following estimate
\begin{equation} \label{NMCI}
\begin{aligned}
\norm{\theta_t (t)}^2_{\Omega_x} + a(\theta (t), \theta (t)) & 
\leq 2 \int^t_0 \norm{\rho_{ss}}_{\Omega_x} \norm{\theta_s}_{\Omega_x} ds \\
&\leq 2 \int^t_0 \norm{\rho_{ss}}_{\Omega_x} ds \max_{s \in [0, T]} \norm{\theta_t}_{\Omega_x} \\
&\leq 2 \left( \int^t_0 \norm{\rho_{ss}}_{\Omega_x} ds \right)^2 
+ \frac{1}{2} \left( \max_{s \in [0, T]} \norm{\theta_t}_{\Omega_x} \right)^2.
\end{aligned}
\end{equation}
Since this holds for all $ t \in [0, T] $ and 
$ a(\theta (t), \theta (t)) \geq 0 $, we have
\[
\frac{1}{2} \left( \max_{s \in [0, T]} \norm{\theta_t}_{\Omega_x}
 \right)^2 \leq 2 \left( \int^T_0 \norm{\rho_{ss}}_{\Omega_x} ds \right)^2.
\]
Inserting this into \eqref{NMCI} and using Lemma \ref{projlemma} leads to
\[
\norm{\theta_t (t)}^2_{\Omega_x} + a(\theta (t), \theta (t)) \leq 4 \left( \int^T_0 \norm{\rho_{ss}}_{\Omega_x} ds \right)^2
\leq 4 \left( C h^2 \int^T_0 \norm{E_{ss}}_{H^2(\Omega_x)} ds \right)^2.
\]
It follows that
\begin{equation} \label{NMCII}
\norm{\theta_t (t)}_{\Omega_x} \leq C h^2 \int^T_0 \norm{E_{ss}}_{H^2(\Omega_x)} ds
\end{equation}
and
\begin{equation} \label{NMCIII}
\Tnorm{\theta (t)}_h \leq C h^2 \int^T_0 \norm{E_{ss}}_{H^2(\Omega_x)} ds.
\end{equation}
Next we note that
\[
\begin{aligned}
2 \norm{\theta}_{\Omega_x} \frac{d}{dt} \norm{\theta}_{\Omega_x}
&= \frac{d}{dt} \norm{\theta}^2_{\Omega_x}
= \frac{d}{dt} \int_{\Omega_x} | \theta|^2 dx \\
&= 2 \int_{\Omega_x} \theta \cdot \theta_t dx
\leq 2 \norm{\theta}_{\Omega_x} \norm{\theta_t }_{\Omega_x}.
\end{aligned}
\]
After cancellation and integration we have
\begin{equation} \label{NMCIV}
\norm{\theta (t)}_{\Omega_x} \leq \int^t_0 \norm{\theta_s (s)}_{\Omega_x}\,  ds 
\leq C h^2 t \int^T_0 \norm{E_{ss}}_{H^2(\Omega_x)} ds.
\end{equation}

Now we have the following a priori error estimates theorem.
\begin{theorem}\label{semidiscTh}
Let $ E $ and $ E^h $ be the solutions of \eqref{MaxwellB} and \eqref{NM2}, 
respectively, such that
$E(t),E_t(t)\in H^2(\Omega_x) $
and $E_{tt}\in L_1\left((0,T);H^2(\Omega_x)\right)$.
Then, there exists a positive constant $ C $ such that for $ t \geq 0$,
\[
\begin{aligned}
\norm{E (t) - E^h (t)}_{\Omega_x} &\leq C h^2 \norm{E (t)}_{H^2(\Omega_x)} + C h^2 t \int^T_0 \norm{E_{ss}}_{H^2(\Omega_x)} ds, \\
\norm{E_t (t) - E^h_t (t)}_{\Omega_x} &\leq C h^2 \norm{E_t (t)}_{H^2(\Omega_x)} + C h^2 \int^T_0 \norm{E_{ss}}_{H^2(\Omega_x)} ds, \\
\Tnorm{E (t) - E^h (t)}_h &\leq C h \norm{E (t)}_{H^2(\Omega_x)} + C h^2 \int^T_0 \norm{E_{ss}}_{H^2(\Omega_x)} ds.
\end{aligned}
\]
\end{theorem}
The proof follows from \eqref{NMCII}-\eqref{NMCIV} together with Lemma \ref{projlemma}.
\begin{remark}
For the magnetic field $ B $ we get a slightly different system of equations,
but the same error estimates will hold.
\end{remark}

\subsection{Time discretization}

Let $ \{ t_m \}_{m=0}^{M} $ be a uniform partition of $ [0, T ] $ 
of step size $ k = T/M $.
Before formulating the fully discrete problem, 
we introduce the following notations of difference quotients 
\[
\begin{aligned}
\dbart u^m &= \frac{u^m - u^{m-1}}{k}, \\
\dbartt u^m &= \frac{u^m - 2 u^{m-1} + u^{m-2}}{k^2}, \\
\hat{u}^m &= \frac{u^m + 2 u^{m-1} + u^{m-2}}{4},
\end{aligned}
\]
where $ u^m = u(t_m ) $. Then a fully discrete problem reads as follows: 
for $ m = 2, 3, \ldots, M $, find $ \EE^m $ such that
\begin{equation} \label{TDNM0}
(\dbartt \EE^m, \chi) + a(\hat{\EE}^m, \chi) = -(j_t^{m - 1}, \chi) \quad \forall \chi \in V^x_h.
\end{equation}
The choices of the first two approximations 
$ \EE^0 $ and $ \EE^1 $ will be discussed later.
We split the error as 
\[
e^m = E^m - \EE^m = (E^m - Q_h E^m) + (Q_h E^m - \EE^m) = \rho^m + \theta^m.
\]
We use Lemma \ref{projlemma} to estimate $ \rho^m $,
hence it remains to estimate $ \theta^m $.
To do so we note that
\[
\begin{aligned}
( \dbartt \theta^m, \chi)  + a( \hat{\theta}^m, \chi) &=
( \dbartt Q_h E^m, \chi) + a(Q_h \hat{E}^m, \chi) - 
( \dbartt \EE^m, \chi) - a( \hat{\EE}^m, \chi) \\
&= ( \dbartt Q_h E^m, \chi) + a(Q_h \hat{E}^m, \chi) - (j_t^{m - 1}, \chi) \\
&= ( \dbartt Q_h E^m, \chi) + a(Q_h \hat{E}^m, \chi) -
(E_{tt}^{m - 1}, \chi) - a(E^{m - 1}, \chi) \\
&= ( \omega^m, \chi) + \frac{k^2}{4} a(\dbartt Q_h E^m, \chi),
\end{aligned}
\]
where $ \omega^m = \dbartt Q_h E^m - E_{tt}^{m - 1} $.
Choose
\[
\chi = \theta^m - \theta^{m - 2} = k(\dbart \theta^m + \dbart \theta^{m - 1})
= (\theta^m + \theta^{m - 1}) - ( \theta^{m - 1} + \theta^{ m -2}).
\]
Then we have
\begin{equation}
\begin{aligned} \label{TDNM1}
( \dbartt \theta^m, \theta^m - \theta^{m - 2})  + a( \hat{\theta}^m, \theta^m - \theta^{m - 2})
=&\frac{1}{k} (\dbart \theta^m - \dbart \theta^{m - 1}, \theta^m - \theta^{m - 2}) \\
&+ a(\hat{\theta}^m, \theta^m - \theta^{m - 2}) \\
=& ( \omega^m, \theta^m - \theta^{m - 2}) \\
&+ \frac{k^2}{4} a(\dbartt Q_h E^m, \theta^m - \theta^{m - 2}).
\end{aligned}
\end{equation}
We define the discrete energy
\[
\bbE^m = \norm{\dbart \theta^m}^2_{\Omega_x} + 
\frac{1}{4} a(\theta^m + \theta^{m - 1}, \theta^m + \theta^{m - 1}).
\]
Now, if in the left hand side of \eqref{TDNM1}
we use the second form of $ \chi $ in the first term and the third form in the second term,
and in the right hand side we use the second form of $ \chi $ in both terms, then we get
\[
\bbE^m - \bbE^{m - 1} = k ( \omega^m, \dbart \theta^m + \dbart \theta^{m - 1})
+ \frac{k^3}{4} a(\dbartt Q_h E^m, \dbart \theta^m + \dbart \theta^{m - 1}).
\]
We estimate the right hand side
using the continuity of $a(\cdot, \cdot)$, with $C_a=(9+\gamma)$,
and the inverse inequality for the triple norm to get
\[
\begin{aligned}
\bbE^m - \bbE^{m-1} \leq & k \norm{\omega^m}_{\Omega_x} \left( \norm{\dbart \theta^m}_{\Omega_x} 
+ \norm{\dbart \theta^{m - 1}}_{\Omega_x} \right)\\
&+ C_a\frac{k^3}{4} \Tnorm{\dbartt Q_h E^m}_h \left( \Tnorm{\dbart \theta^m}_h + \Tnorm{\dbart \theta^{m - 1}}_h \right) \\
\leq & k \norm{\omega^m}_{\Omega_x} \left( \sqrt{\bbE^m} + \sqrt{\bbE^{m - 1}} \right)\\
&+ C_a\frac{k^3 h^{-2}}{4} \norm{\dbartt Q_h E^m}_{\Omega_x}
\left( \norm{\dbart \theta^m}_{\Omega_x} + \norm{\dbart \theta^{m - 1}}_{\Omega_x} \right) \\
\leq& \left( k \norm{\omega^m}_{\Omega_x} + C_a\frac{k^3 h^{-2}}{4}
 \norm{\dbartt Q_h E^m}_{\Omega_x} \right) \left( \sqrt{\bbE^m} + \sqrt{\bbE^{m - 1}} \right).
\end{aligned}
\]
After cancellation it follows that
\[
\sqrt{\bbE^m} \leq \sqrt{\bbE^{m - 1}} + k \norm{\omega^m}_{\Omega_x}
+ C_a\frac{k^3 h^{-2}}{4} \norm{\dbartt Q_h E^m}_{\Omega_x}.
\]
Iterating the above inequality leads to
\begin{equation} \label{TDNM2}
\sqrt{\bbE^m} \leq \sqrt{\bbE^1} + k \sum_{j = 2}^m \norm{\omega^j}_{\Omega_x} + 
C_a\frac{k^3 h^{-2}}{4} \sum_{j = 2}^m \norm{\dbartt Q_h E^j}_{\Omega_x}.
\end{equation}
Now we estimate the terms on the right hand side.
Let us begin with $ \omega^j $ and split it as
\[
\omega^j = (Q_h - I) \dbartt E^j + (\dbartt E^j - E_{tt}^{j-1}) =: 
\omega_1^j + \omega_2^j.
\]
We write $ \omega_1^j $ in the following way
\[
\omega_1^j = \frac{1}{k^2} (Q_h - I) (E^j - 2 E^{j - 1} + E^{j - 2}) = 
\frac{1}{k^2} (Q_h - I) \left( \int_{t_{j - 1}}^{t_j} E_t \, dt - 
\int_{t_{j - 2}}^{t_{j - 1}} E_t \, dt \right).
\]
Summing over $ j $ and using Lemma \ref{projlemma} gives
\[
k \sum_{j = 2}^m \norm{\omega_1^j}_{\Omega_x} \leq \frac{1}{k} 
\sum_{j = 2}^m \int_{t_{j - 2}}^{t_j} \norm{(Q_h - I) E_t}_{\Omega_x} \, dt
\leq \frac{2C h^2}{k} \int_0^{t_m} \norm{E_t}_{H^2 (\Omega_x)} \, dt.
\]
As for $ \omega_2^j $ we use Taylor expansion of  $ E^j $ and $ E^{j - 2} $ 
in  polynomials of degree 2 about $ t_{j - 1} $.
Then, due to cancellations, we end up with 
\[
\omega_2^j = 
\frac{1}{6 k^2} \left( \int_{t_{j - 1}}^{t_j} (t - t_{j - 1})^2 E_{ttt} \, dt 
- \int_{t_{j - 2}}^{t_{j - 1}} (t - t_{j - 1})^2 E_{ttt} \, dt \right).
\]
Once again summing over $ j $ leads to
\[
k  \sum_{j = 2}^m \norm{\omega_2^j}_{\Omega_x} \leq \frac{1}{6 k}  
\sum_{j = 2}^m \int_{t_{j - 2}}^{t_j} (t - t_{j - 1})^2 \norm{E_{ttt}}_{\Omega_x} \, dt
\leq \frac{k}{3} \int_0^{t_m} \norm{E_{ttt}}_{\Omega_x} \, dt.
\]
The third term in \eqref{TDNM2} will be estimated as follows
\[
\begin{aligned}
\frac{k^3 h^{-2}}{4} \sum_{j = 2}^m \norm{\dbartt Q_h E^j}_{\Omega_x} &\leq
\frac{k^3 h^{-2}}{4} \sum_{j = 2}^m \norm{\omega^j}_{\Omega_x} + \norm{E_{tt}^{j - 1}}_{\Omega_x}.
\end{aligned}
\]
It remains to estimate $ \bbE^1 $, which depends on how $ \EE^0 $ and 
$ \EE^1 $ are chosen.
Let $ \EE^0 = Q_h E^0 $ and assume that $ \EE^1 $ 
is chosen such that $ \norm{\theta^1}_{\Omega_x} \leq C(h^2 + k^2) $.
Then we have
\begin{equation*}
\begin{aligned}
\sqrt{\bbE^1} = \sqrt{\norm{\frac{\theta^1}{k}}^2_{\Omega_x} 
+ \frac{1}{4} a(\theta^1, \theta^1)}
&\leq \frac{1}{k} \norm{\theta^1}_{\Omega_x} + \frac{1}{2} \Tnorm{\theta^1}_h\\
&\leq \left( \frac{1}{k} + \frac{1}{2 h} \right) 
\norm{\theta^1}_{\Omega_x} 
\leq C \left( \frac{1}{k} + \frac{1}{2 h} \right) (h^2 + k^2).
\end{aligned}
\end{equation*}
We combine the above estimates to get $ \norm{\dbart \theta^m}_{\Omega_x}
\leq C(h + k) $ for $ m = 2, 3, \ldots, M $,
assuming that $k$ is proportional to $h$.
Finally, if we use the following estimate
\[
\norm{\theta^m}_{\Omega_x} \leq \norm{\theta^1} _{\Omega_x}
+ k\sum_{j=2}^m \norm{\dbart \theta^j}_{\Omega_x},
\]
then we also have that $ \norm{\theta^m}_{\Omega_x} \leq C(h + k) $.
We sum up the results in the following theorem.
\begin{theorem}
Let $ E $ and $ \EE^m $ be the solutions of 
\eqref{MaxwellB} and \eqref{TDNM0}, respectively.
Under the assumptions of Theorem \ref{semidiscTh},
assuming moreover  that $ E_t \in L_1\left((0,T); H^2 ( \Omega_x)\right) $, 
$ E_{ttt} \in L_1 \left((0,T); L_2(\Omega_x)\right) $ 
and that $ \EE^1 $ is chosen such that $ \norm{\theta^1}_{\Omega_x} \leq C(h^2 + k^2) $,
we have the following error estimate
\[
\norm{E(t_m) - \EE^m}_{\Omega_x} \leq C(h + k).
\]
\end{theorem}

\begin{remark}
If we assume extra regularity on $ E $ we can prove that
\[
\norm{E(t_m) - \EE^m}_{\Omega_x} \leq C(h^2 + k^2).
\]
More precisely, we need to assume that
$ E_{tt} \in L_1\left((0,T); H^2 ( \Omega_x)\right) $, 
$ E_{tttt} \in L_1 \left((0,T); L_2(\Omega_x)\right) $ 
and that $ \EE^1 $ is chosen such that $ \norm{\theta^1}_{\Omega_x} \leq C(h^3 + k^3) $.
\end{remark}

\section{Numerical Results}

Here we present some numerical results justifying the accuracy of our method.
We performed the calculations for the simplified case of one space variable
and two velocities variables, i.e. the one and one-half dimensional
Vlasov-Maxwell system (cf \cite{Standar:2016}), which takes the following form:
\[
\begin{array}{c}
\d_t f + v_1 \d_x f + (E_1 + v_2 B)\d_{v_1} f + (E_2-v_1 B)\d_{v_2} f=0, \\
\d_t E_1 = - \int v_1 f dv= -j_1(t,x), \\
\d_t E_2 +\d_x B = -\int v_2 f dv=-j_2(t,x), \\
\d_t B + \d_x E_2 =0,
\end{array}
\]
where $f=f(t,x,v_1,v_2)$, $E_1=E_1(t,x)$, $E_2=E_2(t,x)$, $ B=B(t,x)$
with $x\in\Omega_x\subset\R$ and $v=(v_1,v_2)\in\Omega_v\subset\R^2$.
We assume here the non-relativistic case of the Vlasov-Maxwell system,
since there is a wider literature available to compare the numerical test with.
We note that our theoretical results are also valid in this case.

The initial conditions are given by
\[
\begin{array}{c}
f(0,x,v_1,v_2) = \frac{1}{\pi\beta}e^{-v_1^2/\beta}
\left[\mu e^{-(v_2-v_{0,1})^2/\beta}+(1-\mu)e^{-(v_2+v_{0,2})^2/\beta}\right], \\
E_1(0,x) = E_2(0,x) = 0, \quad B(0,x) = -b\sin(k_0x_1),
\end{array}
\]
which corresponds to the streaming Weibel instability (cf \cite{Che})
with $\beta = 0.01$ and $b=0.001$.
We perform the calculations for two sets of values of parameters:
\[
\begin{array}{ll}
\mbox{case 1:} & \mu=0.5,\ v_{0,1}=v_{0,2}=-0.3,\ k_0=0.2,\\
\mbox{case 2:} & \mu=1/6,\ v_{0,1}=-0.5,\ v_{0,2}=-0.1,\ k_0=0.2,\\
\end{array}
\]
with $x\in[0,L]$, $L=2\pi/k_0$.
Periodic boundary condition is assumed for $x$ variable, which
we normalized in our computations taking $x\in[0,1]$.
For the accuracy test we set $\Omega_v=[-1,1]^2$,
whereas for the other test we set $\Omega_v=[-1.1,1.1]^2$.

\subsection{Accuracy tests}

The Vlasov-Maxwell system is reversible in time for the above.
Thus denoting the initial conditions as $f(0,x,v)$, $E(0,x)$, $B(0,x)$,
we get at time $t=T$ the solution $f(T,x,v)$, $E(T,x)$, $B(T,x)$.
Now, taking $f(T,x,-v)$, $E(T,x)$, $-B(T,x)$ as the initial solution at $t=0$,
we recover $f(0,x,-v)$, $E(0,x)$, $-B(0,x)$ at $t=T$.

Using this theoretical fact we run the calculation for $T=5$
and show the $L_1$ and $L_2$ errors of solutions
for several choices of degree of polynomials $p$ and mesh parameters
$h_t$, $h_x$, $h_v$.
For all calculations we used the uniform degrees $p$ in all cells of uniform meshes.
We present the results for the following choice of mesh sizes sets:
$H_1$ corresponds to $h_t=h_x=0.1$ and $h_v=\sqrt{2}/6$;
$H_2$ corresponds to $h_t=h_x=0.05$ and $h_v=\sqrt{2}/12$;
$H_3$ corresponds to $h_t=h_x=0.025$ and $h_v=\sqrt{2}/24$.

Table \ref{table1} lists the errors for the fixed mesh set $H_1$
and increasing degree of finite elements polynomial approximation,
whereas in Table \ref{table2} we list the errors for the fixed degree $p=1$
of polynomial approximation and different mesh sizes.

\begin{table}[!h]
\centering
\caption{$L_1$ and $L_2$ errors for different polynomial degrees and fixed
mesh sizes set $H_1$.}
\begin{tabular}{cccccc}
\hline
Error & Degree & $f$ & $E_1$ & $E_2$ & $B$ \\
\hline
$L_1$ & $p=1$ & 3.801e-1 & 7.086e-4  & 1.599e-6  & 1.645e-5 \\
      & $p=2$ & 1.614e-1 & 3.248e-9  & 1.770e-7  & 9.092e-7 \\
      & $p=3$ & 1.891e-2 & 2.295e-10 & 9.753e-9  & 3.321e-8 \\
\hline
$L_2$ & $p=1$ & 7.302e-1 & 6.204e-7  & 3.517e-12 & 4.303e-10 \\
      & $p=2$ & 1.632e-1 & 1.498e-17 & 4.113e-14 & 1.070e-12 \\
      & $p=3$ & 2.833e-3 & 6.648e-20 & 1.185e-16 & 2.186e-15 \\
\hline
\end{tabular}
\label{table1}
\end{table} 

\begin{table}[!h]
\centering
\caption{$L_1$ and $L_2$ errors for different mesh sizes and fixed
polynomial degree $p=1$.}
\begin{tabular}{cccccc}
\hline
Error & Mesh sizes set & $f$ & $E_1$ & $E_2$ & $B$ \\
\hline
$L_1$ & $H_1$ & 3.801e-1 & 7.086e-4  & 1.599e-6  & 1.645e-5 \\
      & $H_2$ & 1.629e-1 & 8.304e-10 & 1.791e-7  & 8.387e-6 \\
      & $H_3$ & 4.324e-2 & 2.016e-10 & 4.750e-8  & 2.099e-6 \\
\hline
$L_2$ & $H_1$ & 7.302e-1 & 6.204e-7  & 3.517e-12 & 4.303e-10 \\
      & $H_2$ & 1.939e-1 & 8.520e-19 & 3.956e-14 & 9.298e-11 \\
      & $H_3$ & 1.444e-2 & 5.014e-20 & 2.784e-15 & 5.850e-12 \\
\hline
\end{tabular}
\label{table2}
\end{table} 

We can see from the tables the convergence of our method for all functions.
We present the results for one value of the stability parameter $\delta=0.05$,
since its choice does not influence importantly (in some reasonable interval of values)
the accuracy, but only the stability of the method.

\subsection{Streaming Weibel instability tests}

In this section we present the preliminary results
for the streaming Weibel instability tests. More results will be included
in the forthcoming paper \cite{Malmberg_Standar}.

We present the time evolution of the magnetic, electric and kinetic energies
for both cases of parameters values.
The calculations were carried out for the mesh sizes
$h_t=1/20$, $h_x=1/30$, $h_v=\sqrt{2}\cdot 11/300$ with $p=1$.
We plot the components of electric energy $E_i=\frac{1}{2L}\int_0^L E_i^2\,dx_1$,
$i=1,2$, and the magnetic energy $B=\frac{1}{2L}\int_0^L B^2\,dx_1$ in Figure~\ref{fig1}.
The kinetic energy is showed in Figure~\ref{fig2} as the separate components
defined by $K_i=\frac{1}{2L}\int_0^L\int_{\Omega_v} v_i^2 f \,dvdx_1$, $i=1,2$.
The qualitative behaviour of the time evolution of both the electromagnetic
and kinetic energies is in agreement with the theory
and the results presented in \cite{Che} for different numerical methods.

\begin{figure}[!h]
\centerline{
\includegraphics[width=0.5\textwidth]{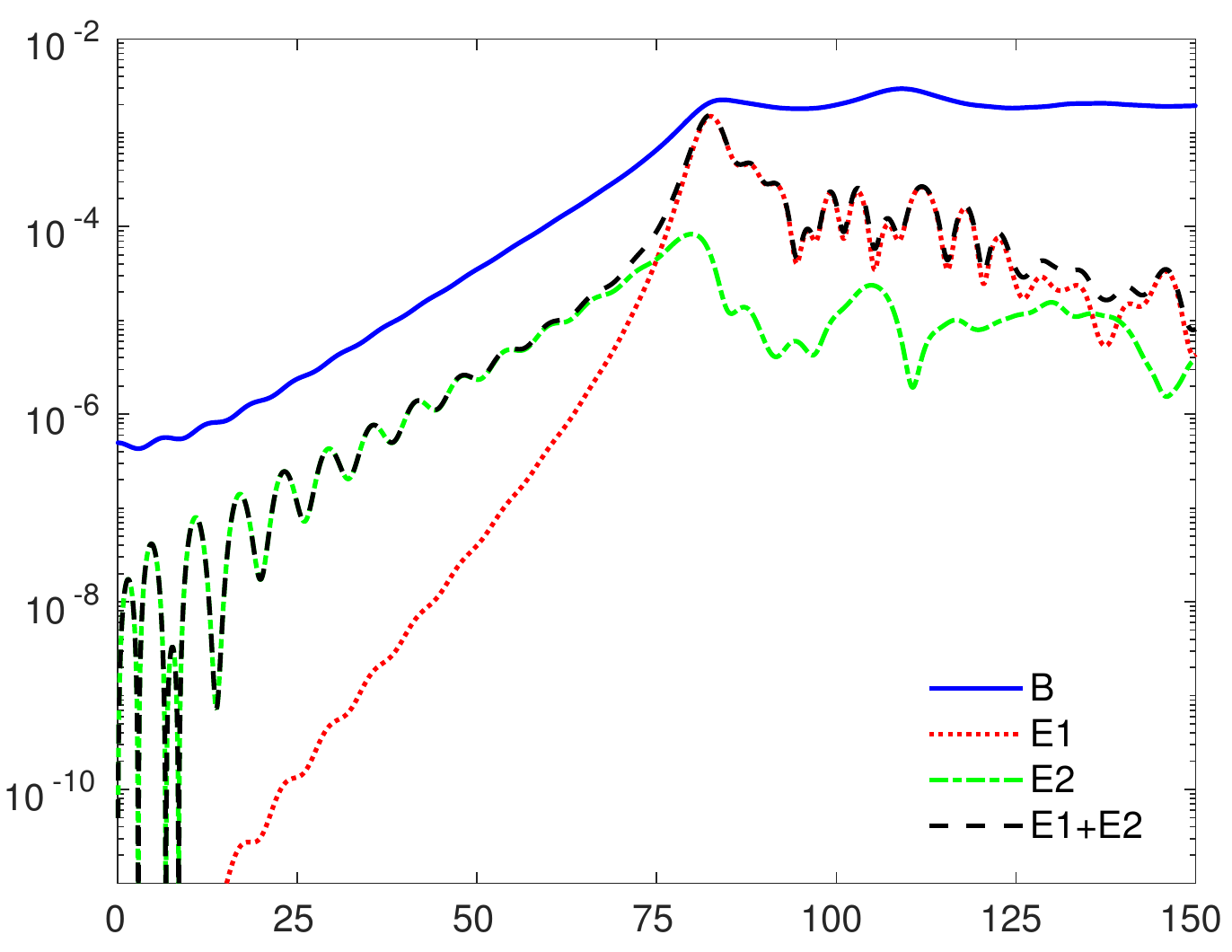}
\includegraphics[width=0.5\textwidth]{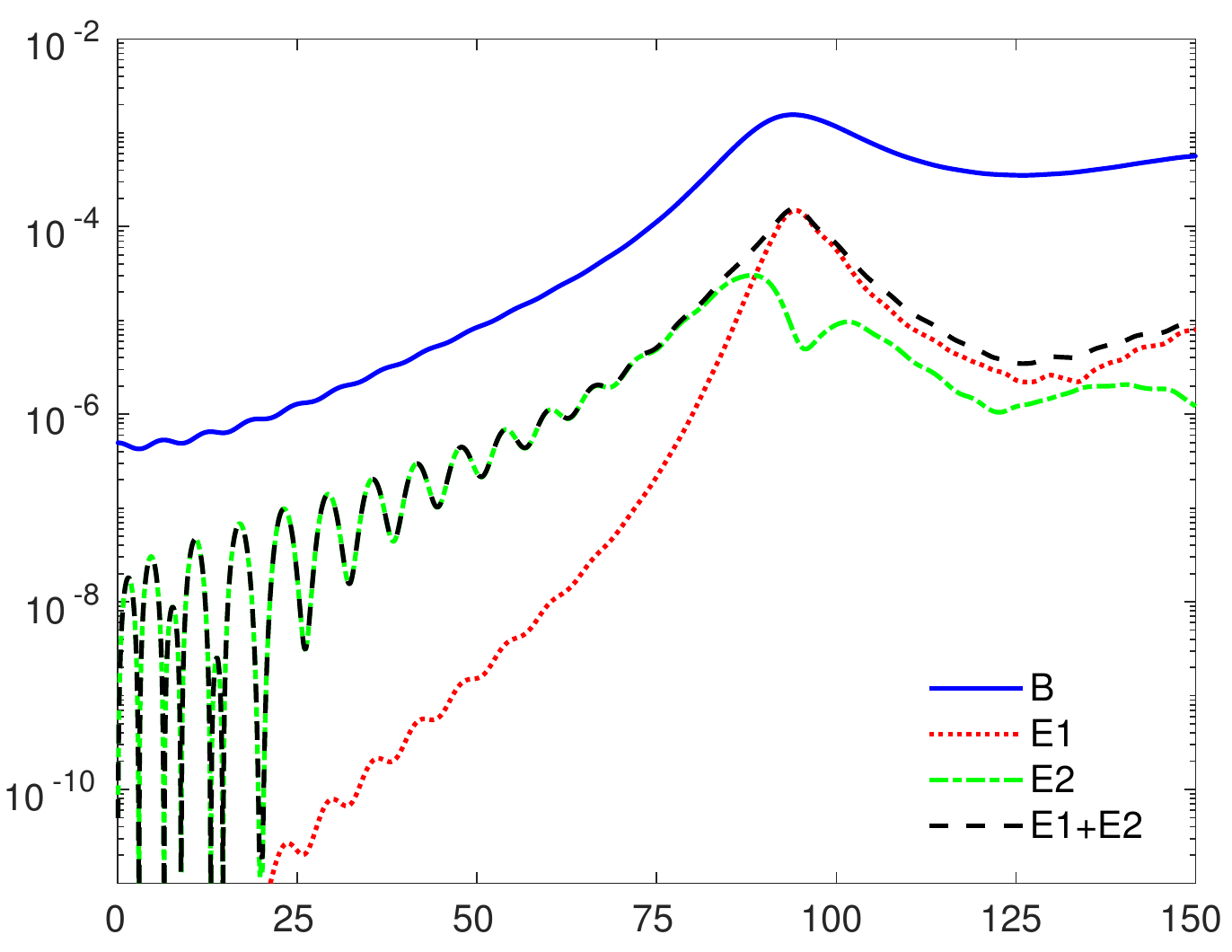}
}
\caption{Magnetic ($B$) and electric ($E_1$, $E_2$) energy for case~1 (left)
and case 2 (right).}
\label{fig1}
\end{figure}

\begin{figure}[!h]
\centerline{
\includegraphics[width=0.5\textwidth]{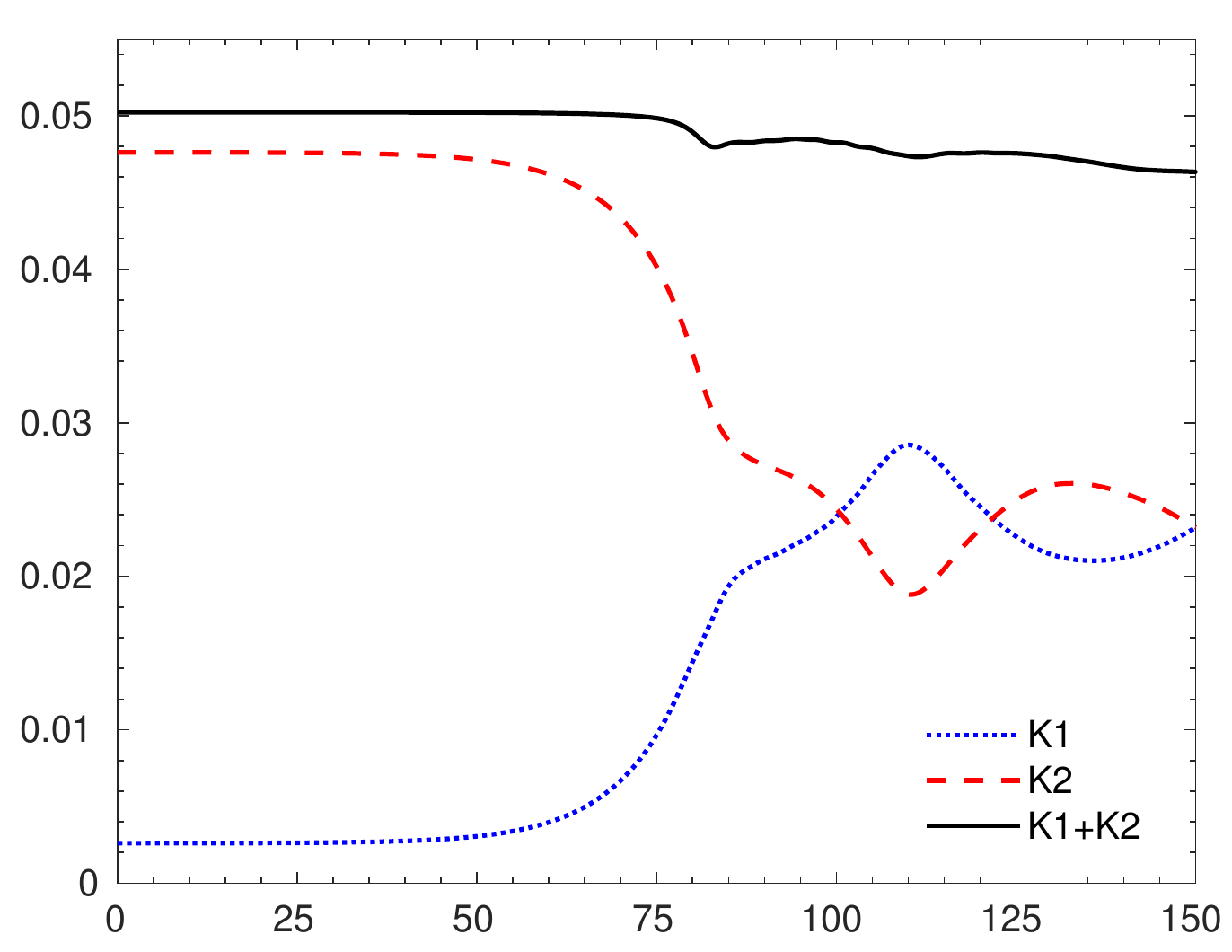}
\includegraphics[width=0.5\textwidth]{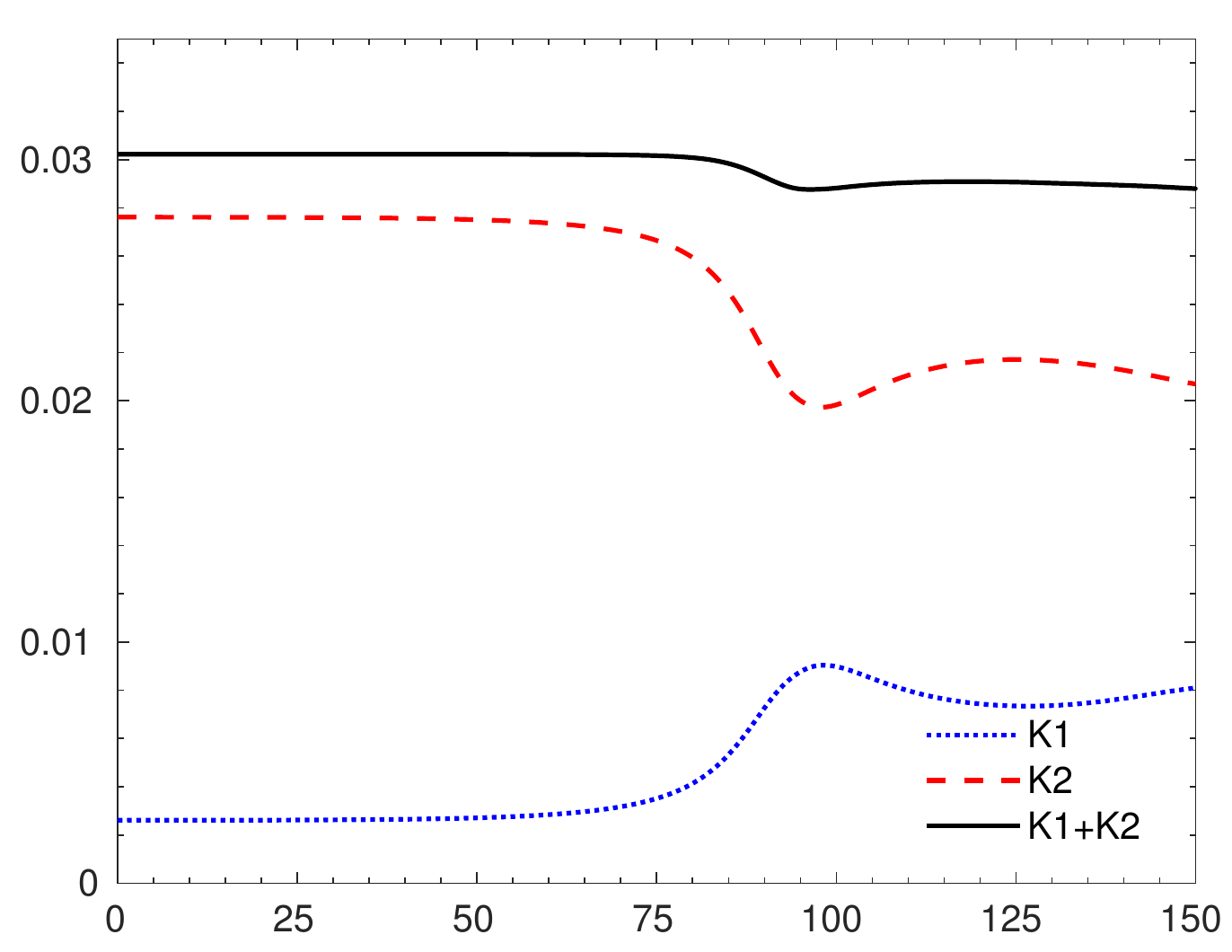}
}
\caption{Kinetic energy for case 1 (left) and case 2 (right).}
\label{fig2}
\end{figure}

\section{Conclusion} 
This paper concerns two approaches in the numerical investigation for the 
Vlasov-Maxwell system.  
The first study is devoted to the $hp$ streamline diffusion method for the 
relativistic Vlasov-Maxwell system in fully 3-dimensions. 
Our objective is to present unified phase-space (for Maxwell's equations) 
and phase-space-time (for the Vlasov-Maxwell system) discretization schemes  
that have optimal order convergence for the hyperbolic problems 
( ${\mathcal O}(h^{s-1/2})$ for solutions in the Sobolev space $H^s(\Omega)$ ) 
with strong stability properties and adaptivity features. 
The adaptivity in a priori regime is based on refining in the vicinity of 
singularities combined with lower order approximating polynomials and 
non-refined mesh with higher spectral order in smooth regions. 
In this way we have constructed a finite 
element mesh with several improving properties, e.g. 
stability, convergence, and adaptivity, gathered in it. To our knowledge, 
except in some work in convection-diffusion problems, see e.g. 
\cite{Suli_etal} and our study in \cite{AsadSopas}, such approach 
is not considered for this type of equations elsewhere. 

The second study concerns a penalty method for the Maxwell's equations, which 
is based on a certain Nitsche type symmetrization scheme. In this part we have
combined the field equations to a second order pde. For this equation 
 we derive a second order spatial approximation for the Nitsche's scheme. 
We also prove a second order temporal discretization, 
 assuming a somewhat more regular-in-time field functions. Even this approach 
is not considered in any other works for the VM system. 

The results are justified in lower dimensional cases through the accuracy
and the streaming Weibel instability tests presented in this paper
and through implementing some numerical examples in the forthcoming paper,
see \cite{Malmberg_Standar}. 
The full-dimensions are to expensive to experiment. 
However, the theoretical analysis and numerical justifications in 
low dimensions are indicating the robustness of the considered schemes.

\section*{Acknowledgment} 

The research of the first author was supported by the 
Swedish Research Council VR.

\def\listing#1#2#3{{\sc #1}:\ {\it #2},\ #3.}

\end{document}